\documentclass{amsart}

\usepackage[a4paper, left=1.25in,right=1.25in]{geometry}

\usepackage{mathrsfs}
\usepackage{amsfonts}
\usepackage{amsmath}
\usepackage{amssymb}
\usepackage{mathtools}
\usepackage{float}
\usepackage{xcolor}
\usepackage[inline]{enumitem}
\usepackage{mathdots}
\usepackage{csquotes}
\usepackage{subcaption}
\usepackage{mathdots}
\usepackage{tikz}
\usepackage{tikz-cd}
\usetikzlibrary{arrows,automata}
\usetikzlibrary{decorations.markings}
\usetikzlibrary{automata, positioning, arrows}
\usetikzlibrary{math}
\usepackage{hyperref}
\usepackage{centernot}
\usepackage{varwidth}
\usepackage{bbm}

\newcommand{\pres}[3]{\textnormal{#1} \langle #2 \mid #3 \rangle}

\newcommand{\EDTOL}{\mathbf{EDT0L}}

\newtheorem{innercustomgeneric}{\customgenericname}
\providecommand{\customgenericname}{}
\newcommand{\newcustomtheorem}[2]{%
  \newenvironment{#1}[1]
  {%
   \renewcommand\customgenericname{#2}%
   \renewcommand\theinnercustomgeneric{##1}%
   \innercustomgeneric
  }
  {\endinnercustomgeneric}
}
\newtheorem{theorem}{Theorem} 
\newtheorem*{theorem*}{Theorem} 
\newcustomtheorem{customthm}{Theorem}
\newcustomtheorem{customcor}{Corollary}
\newcustomtheorem{customprop}{Proposition}

\newtheorem{lemma}[theorem]{Lemma}     
\newtheorem{corollary}[theorem]{Corollary}

\newtheorem{proposition}[theorem]{Proposition}
\numberwithin{theorem}{section}

\theoremstyle{definition}

\newtheorem{question}{Question}
\newtheorem*{question*}{Question}
\newtheorem{example}{Example}
\newtheorem{remark}{Remark}
\newtheorem{conjecture}{Conjecture}

\newcommand{\raag}{{\textsc{raag}}}
\newcommand{\raags}{{\textsc{raag}\textnormal{s}}}
\newcommand{\rabsag}{{\textsc{rabsag}}}
\newcommand{\rabsags}{{\textsc{rabsag}\textnormal{s}}}
\DeclareMathOperator{\BS}{BS}
\DeclareMathOperator{\NP}{NG}

\begin{document}

\title[]{On the Diophantine Problem in Some One-relator Groups}

\author{Carl-Fredrik Nyberg-Brodda}
\address{Alan Turing Building, Department of Mathematics, University of Manchester, UK.}
\email{carl-fredrik.nybergbrodda@manchester.ac.uk}

\thanks{The author wishes to thank the Dame Kathleen Ollerenshaw Trust, which funded his position as Research Associate at the University of Manchester while this research was carried out.}

\subjclass[2020]{20F05 (primary); 20F10, 20F32, 20M05(secondary)}

\date{\today}


\keywords{}

\begin{abstract}
We study the Diophantine problem, i.e. the decision problem of solving systems of equations, for some families of one-relator groups, and provide some background for why this problem is of interest. The method used is primarily the Reidemeister--Schreier method, together with general recent results by Dahmani \& Guirardel and Ciobanu, Holt \& Rees on the decidability of the Diophantine problem in general classes of groups. First, we give a sample of the methods of the article by proving that the one-relator group with defining relation $a^mb^n = 1$ is virtually a direct product of hyperbolic groups for all $m, n \geq 0$, and thus conclude decidability of the Diophantine problem in such groups. As a corollary, we obtain that the Diophantine problem is decidable in any torus knot group. Second, we study the two-generator, one-relator groups $G_{m,n}$ with defining relation a commutator $[a^m, b^n] = 1$, where $m, n \geq 1$. In doing so, we define and study a natural class of groups (\rabsags{}), related to right-angled Artin groups (\raags{}). We reduce the Diophantine problem in the groups $G_{m,n}$ to the Diophantine problem in groups which are virtually certain \rabsags{}. As a corollary of our methods, we show that the submonoid membership problem is undecidable in the group $G_{2,2}$ with the single defining relation $[a^2, b^2] = 1$. We use the recent classification by Gray \& Howie of \raag{} subgroups of one-relator groups to classify the \raag{} subgroups of some \rabsags{}, showing the potential usefulness of one-relator theory to this area. Finally, we define and study Newman groups $\NP(p,q)$, which are $(p+1)$-generated one-relator groups generalising the solvable Baumslag--Solitar groups. We show that all such groups are hyperbolic, and thereby also conclude decidability of their Diophantine problem. 
\end{abstract}

\maketitle

\noindent One of the main characters of combinatorial group theory in the past century is undoubtedly the class of one-relator groups, i.e. those groups which can be presented with a single defining relation. Arising as the fundamental group of $2$-manifolds and complements of torus knots, one-relator groups were a natural link between the topological origins of combinatorial group theory and its subsequent topology-free developments. For example, the word problem for groups, introduced by Dehn \cite{Dehn1911}, was motivated by the problem of deciding whether a given simple loop on an orientable $2$-manifold $S_g$ (of genus $g > 1$) can be contracted to a point or not -- indeed, this problem is equivalent to the algebraic solution of the word problem in the fundamental group of $S_g$ in question, which Dehn proved is always a one-relator group
\begin{equation}\label{Eq:surface-group}
\pi_1(S_g) \cong \pres{Gp}{a_1, b_1, \dots, a_g, b_g}{[a_1, b_1][a_2, b_2] \cdots [a_g, b_g] = 1},
\end{equation}
where $[a_i, b_i] = a_i b_i a_i^{-1} b_i^{-i}$ denotes the commutator ($\pi_1(S_g)$ is often called a \textit{surface} group). Dehn solved this word problem, and so the first non-trivial (i.e. non-free, non-abelian, infinite) example of groups with decidable word problem were one-relator groups. Ever since, one-relator groups have consistently appeared in any treatment on combinatorial group theory, see e.g. the books by Kurosh \cite[p. 271--272 (Russian)]{Kurosh1953}; Adian \cite[p. 5 (Russian)]{Adian1966}; Magnus, Karrass \& Solitar \cite[Chapter~IV]{Magnus1966}; and Lyndon \& Schupp \cite[II.\S5--6, III.\S9 \& IV.\S5]{Lyndon1977}.\footnote{In 1982, Chandler \& Magnus \cite[p.~121]{Chandler1982} commented ``our knowledge of one-relator groups is by now broad enough to justify a monograph dealing exclusively with this topic''. No such monograph has yet appeared.} In part, this is because they one-relator groups are, from the combinatorial viewpoint afforded by group presentations, one of the most natural classes of non-free groups. However, as we shall see in the present article, one-relator groups are of a much more rich, mysterious, and wondrous nature than their simple presentations may suggest.

The first success in one-relator group theory comes in the form of Magnus' 1930 proof of the \textit{Freiheitssatz} (Ger. ``freeness theorem'') \cite{Magnus1930}, relying on the Reidemeister--Schreier technique in a central way. The theorem shows, in essence, that many subgroups of one-relator groups are free (see \S\ref{Subsec:OR-groups}). Using the theorem, Magnus proved in his next paper \cite[\S4]{Magnus1931} that the word problem is decidable in one-relator groups with defining relator of the form $a^{\alpha_1}b^{\beta_1}a^{\alpha_2}b^{\beta_2} = 1$, for $\alpha_i, \beta_i \in \mathbb{Z}$. Extending this, Magnus' subsequently proved in 1932 that the word problem is decidable in \textit{all} one-relator groups \cite{Magnus1932}. We remark that this staggering success pre-dates by over two decades Novikov's construction \cite{Novikov1955} of the first known group with undecidable word problem.\footnote{Although it is today taken for granted that such undecidable groups exist, this was by no means a certainty for researchers at the time. For example (see \cite[p. 169]{Collins1985}), two announcements were made in the same week at the University of Manchester --  A. Turing announced: \textit{the word problem for groups is undecidable in general}, while B. H. Neumann announced: \textit{the word problem for groups is decidable} [!] \textit{in general}. Even further, \textit{both} authors shortly thereafter retracted their claims! (Turing could rescue his argument in the case of cancellative semigroups, which was later published as \cite{Turing1950}).} 

The theory of one-relator groups, following Magnus' remarkable results, saw somewhat of a dormant phase for the subsequent 30 years. In part, Baumslag thought (see \cite[p. 119]{Chandler1982}, this may have been due to the belief that Magnus had, in a sense, proved all important theorems that could be proved. In the 1960s, and following an important paper by Karrass, Magnus \& Solitar \cite{Karrass1960}, one-relator group theory saw a resurgence, in large part due to Baumslag. For example, in 1962, he and Solitar \cite{Baumslag1962} proved the existence of a non-Hopfian one-relator group (see below), and Baumslag produced a number of papers on the subject \cite{Baumslag1964, Baumslag1967, Baumslag1968, Baumslag1968b, Baumslag1969, Baumslag1971}; see also e.g. \cite{Murasugi1964, Fischer1972} from the same time. One of the key findings of this period was the discovery of a clear dividing line between two distinct classes of one-relator groups: the division of one-relator groups into those with torsion and those which were torsion-free. Coarsely speaking, more and more difficulties and unexpected behaviours appeared in the torsion-free case, while the torsion case (somewhat surprisingly) emerged as in many ways behaving analogously to free groups. In this latter category, the most noteworthy contribution comes from B. B. Newman, a student of Baumslag's, who proved in his Ph.D. thesis \cite{Newman1968} his celebrated ``Spelling Theorem'' (see \cite{NybergBrodda2021a} for the remarkable story behind this theorem). In modern language, this theorem has as corollaries that one-relator groups are (Gromov) hyperbolic and have a decidable conjugacy problem.

Generalising the conjugacy problem for groups is the \textit{Diophantine problem}, which is the second main character of the present article. Let $G$ be a group finitely generated (as a monoid) by a set $A$, and let $X$ be a finite set of formal variables. An \textit{equation} over $G$ is a word over $(A \cup X)^\ast$, and a \textit{system} of equations is a finite set of equations $w_i$. A \textit{solution} $\sigma$ to a system of equations $w_1, w_2, \dots, w_n$ is a homomorphism $\sigma \colon (A \cup X)^\ast \to G$ with the property that (1) $\sigma(w_i) = 1$ for all $1 \leq i \leq n$, and (2) $\sigma(a) = a$ for all $a \in A$. That is, informally speaking, given
\begin{equation}\label{Eq:Sample-eq}
XgXhYg^{-1}X^{-1}h^{-1}Y^{-1}
\end{equation}
where $g, h$ are words in the generators of $G$, and $X, Y$ are variables, then solving this equation simply means finding values in $G$ for $X$ and $Y$ which makes the expression \eqref{Eq:Sample-eq} equal to $1$ in $G$. Thus, for example, solving a system of equations over $\mathbb{Z}^n$ is just ordinary linear algebra, and solving systems of equations over finite groups can be done by trial and error. The Diophantine problem is the decision problem which takes as input a system of equations, and outputs a solution $\sigma$ if one exists, and otherwise outputs that no solution exists. As a particular case, if the Diophantine problem is solvable in a group $G$, then we can solve the equation $XgX^{-1}h^{-1} = 1$ for every pair of words $g, h$ in the generating set of $G$; that is, we can solve the conjugacy problem. We remark that there is an obvious generalisation of the Diophantine problem to semigroups and monoids by using pairs of words instead of a single word.

Recently, the Diophantine problem has seen a good deal of study in various classes of groups (see e.g. \cite{Ciobanu2016, Evetts2022, Kharlampovich2020, Garreta2020, Garreta2021}). Most notably, we can trace the following line of development. The Diophantine problem, though initially suspected to be undecidable even for free monoids (cf. e.g. \cite{Matiyasevich1968}), was proved to be decidable for all free monoids in a remarkable paper by Makanin \cite{Makanin1977}. Five years later, Makanin proved the decidability of the problem also in free groups \cite{Makanin1982}. Razborov (also a student of S. I. Adian) later simplified this proof and built the theory of the so-called \textit{Makanin--Razborov diagrams}. The theory which grew out of this work became crucial in the work by Kharlampovich \& Myasnikov (e.g. \cite{Kharlampovich2006}) and Sela (e.g. \cite{Sela2001}) in solving the Tarski problems (far beyond the scope of this article). Via this, recently Dahmani \& Guirardel \cite{Dahmani2010} proved that the Diophantine problem is decidable in all hyperbolic groups. We remark that this result, and more generally the Diophantine problem in groups, has been analysed from the point of view of formal language theory (e.g. $\EDTOL$ languages), see e.g. \cite{Ciobanu2016, Ciobanu2020, Ciobanu2021, Evetts2022}

Thus, in passing from free groups to hyperbolic groups, much work has been done. However, given the natural step from free groups to one-relator groups, comparatively little work has been done to directly attack the following problem:

\begin{question*}
Is the Diophantine problem decidable in every one-relator group?
\end{question*}

The goal of this present article is to demonstrate that this problem can be connected with some interesting new classes of groups, and that many techniques from classical one-relator group theory can be applied to attack this question. As mentioned above, every one-relator group with torsion is hyperbolic as a consequence of the B. B. Newman Spelling Theorem, and hence every such group has decidable Diophantine problem by \cite{Dahmani2010}. Thus, the above question is really a question about torsion-free one-relator groups. In general, even the conjugacy problem (a very particular case of the Diophantine problem) remains open for this class.

Hence, one theme of this present article, in pursuit of partial solutions to the above question, will be an investigation of the subgroup (and, to some extent, submonoid) structure of some one-relator groups. One recently classified class of subgroups of one-relator groups is the following. For any undirected graph $\Gamma$, we may define a \textit{right-angled Artin group} $A(\Gamma)$, which, loosely speaking, have defining relations encoding the edges of the graph $\Gamma$ (see \S\ref{Subsec:RAAG-intro} for details). Right-angled Artin groups, or \raags{} (see \S\ref{Subsec:RAAG-intro}), form a very rich class of groups, particularly due to their subgroup structure and wide applicability. In 2021, as part of his proof of the undecidability of the word problem for one-relator inverse monoids, Gray \cite{Gray2020} gave a classification of the right-angled Artin subgroups of one-relator groups. This classification says that if $\Gamma$ is a graph, then $A(\Gamma)$ embeds in some one-relator group only if $\Gamma$ is a finite forest. Furthermore, if $\Gamma$ is a finite forest, then there exists a fixed one-relator group which contains $A(\Gamma)$.

The author became interested in understanding which one-relator groups are \textit{virtually} right-angled Artin groups after looking at the Diophantine problem in Baumslag-Solitar groups, i.e. the family of groups
\[
\BS(m,n) = \pres{Gp}{a,b}{ba^mb^{-1} = a^n}
\]
where $m, n \in \mathbb{Z}$. These groups, which were introduced in \cite{Baumslag1962}, have, despite their deceptively simple appearance, a number of devious properties -- for example, $\BS(2,3)$ is not Hopfian, and was the first one-relator group known to have this property. The Diophantine problem was recently proved decidable in the (metabelian) groups $\BS(1,n)$, i.e. when $m=1$, see \cite{Kharlampovich2020}, and, with one class of exceptions, remains open for $\BS(m,n)$ in general. The class of exceptions is the class of ``unimodular'' groups $\BS(m,m)$ (as well as $\BS(m,-m)$), and this observation was the starting point for the present investigation. The argument for the decidability of the Diophantine problem in $\BS(m,m)$ is as follows. The group $\BS(m,m)$ has an index $m$ subgroup isomorphic to $F_m \times \mathbb{Z}$. As $F_m \times \mathbb{Z}$ is a direct product of finitely many hyperbolic groups, $\BS(m,m)$ is virtually such a direct product; any group in this latter class has decidable Diophantine problem by a recent result due to Ciobanu, Holt \& Rees \cite{Ciobanu2020}. 

In this article, we will generalise this result for $\BS(m,m)$ in various ways, primarily by using the Reidemeister--Schreier method for computing a presentation for subgroups. First, in \S\ref{Sec:Warmup-torus-knots}, we shall prove as a warm-up that $\pres{Gp}{a,b}{a^mb^n=1}$ is virtually $F_{mn} \times \mathbb{Z}$ for all $m, n \geq 2$, generalising a result of Katayama \cite{Katayama2017}, and hence find the following result: 

\begin{customthm}{\ref{Thm:warmup-am-bn-diophantine}}
For every $m, n \geq 0$ the group $\pres{Gp}{a,b}{a^mb^n=1}$ has decidable Diophantine problem. In particular, the Diophantine problem is decidable in the fundamental group of the complement of any torus knot.
\end{customthm}

Following this, in \S\ref{Sec:Moldavanskii-sec} we will study the Diophantine problem in the Moldavanskii-Tieudjo groups $G_{m,n} = \pres{Gp}{a,b}{[a^m, b^n]=1}$ where $m, n \geq 0$. To do this, we will first in \S\ref{Subsec:RABSAG-intro} introduce a new class of groups: right-angled Baumslag-Solitar-Artin groups $B(\Gamma)$, or \rabsags{} for short. These are defined in a manner similar to right-angled Artin groups, but are instead defined with respect to a directed $\mathbb{Z}$-labelled graph $\Gamma$. We will prove that \rabsags{} can behave very exotically compared to \raags, especially with respect to the submonoid (\S\ref{Subsec:SMMP}) and subgroup (\S\ref{Subsec:SGMP}) membership problems. We conjecture (but do not prove) that the Diophantine problem is decidable in all \rabsags{}, just like it is for all \raags. We study a certain class of \rabsags{} $B(S_{k,\ell})$, defined by an underlying ``star-like'' graph $S_{k,\ell}$, and characterise precisely when the submonoid membership and rational subset membership problems are decidable in such groups. These groups appears to be similarly well-behaved to the \raags{} $F_k \times \mathbb{Z}$. For this reason, we conjecture (see Conjecture~\ref{Conj:all_rabsag_virtually_star_dec_DP}) that, analogous to the recent result by Ciobanu, Holt \& Rees \cite{Ciobanu2020}, the Diophantine problem is decidable in any group that is virtually $B(S_{k,\ell})$. 

Applying this theory to one-relator groups, we manage to prove:

\begin{customthm}{\ref{Thm:[am,bn]-contains-BS-star}}
Let $G_{m,n} = \pres{Gp}{a,b}{[a^m, b^n] = 1}$ with $m, n \geq 1$. Then $G_{m,n}$ contains an index $m$ normal subgroup isomorphic to $B(S_{m,n})$.
\end{customthm}

Consequently, we find that if Conjecture~\ref{Conj:all_rabsag_virtually_star_dec_DP} is true, then $G_{m,n}$ has decidable Diophantine problem for all $m,n \geq 1$. We are able to use the above Theorem~\ref{Thm:[am,bn]-contains-BS-star} for two applications. In the first case, we are able to use our earlier results on membership problems in \raags{} to show:

\begin{customcor}{\ref{Cor:Gmn-has-undec-SMMP}}
The submonoid membership problem is undecidable in the one-relator group
\[
G_{m,n} = \pres{Gp}{a,b}{[a^m, b^n] = 1}
\]
whenever $m \geq 2$ and $n \geq 2$. 
\end{customcor}

In particular, the submonoid membership problem is undecidable in the one-relator group $G_{2,2} = \pres{Gp}{a,b}{[a^2, b^2] = 1}$. This produces a simpler example of a one-relator group with this property than that due to Gray \cite{Gray2020}. Furthermore, unlike that example, the groups $G_{m,n}$ have an additional structural property, yielding:

\begin{customcor}{\ref{Cor:FP-with-commuting-can-be-undec}}
There exists a free product $G$ of two infinite cyclic groups with commuting subgroups such that the submonoid membership problem in $G$ is undecidable. 
\end{customcor}

In \S\ref{Subsec:applying-MT-to-rabsags}, we apply some results on one-relator groups to characterise the right-angled Artin subgroups of a certain class of \rabsags{}.

Finally, in \S\ref{Sec:Torsion-subgroups}, we find embeddings of certain torsion-free one-relator groups into some one-relator groups with torsion, and thereby produce new classes of torsion-free one-relator groups with decidable Diophantine problem. In particular, we study two classes of generalisations of the solvable Baumslag--Solitar groups $\BS(1,k)$. First, generalising an example due to B. B. Newman, we define for $p, q \geq 1$ the \textit{Newman groups}
\begin{equation}
\NP(p,q) = \pres{Gp}{y, x_1, x_2, \dots, x_p}{y = (y^q)^{x_1} (y^q)^{x_2} \cdots (y^q)^{x_p}}\tag{\ref{Eq:NewmanPridegroup-def}}
\end{equation}
Here, we use the convention that $x^y = yxy^{-1}$. As $\BS(1,k) = \NP(1,k)$, this class of groups indeed generalises the solvable Baumslag--Solitar groups. We further define a class one-relator groups which torsion, which we call \textit{inflated} Baumslag--Solitar groups. These are defined in the same way as $\BS(m,n)$, but with the defining relator a power of the relator of $\BS(m,n)$. We prove (Proposition~\ref{Prop:Newman-group-in-inflated}) that $\NP(p,q)$ embeds in $\BS^{(p)}(1,pq)$ as an index $p$ subgroup. We hence obtain:

\begin{customcor}{\ref{Cor:NP-have-dec-DP}}
For every $p, q \geq 1$, the Diophantine problem is decidable in the group
\[
\NP(p,q) = \pres{Gp}{x_1, x_2, \dots, x_n}{x_1 = x_1^{x_2} x_1^{x_3} \cdots x_1^{x_n}}.
\]
\end{customcor}

Finally, we remark that the present article is not intended as a comprehensive treatment of the full extent of the methods herein, but rather as a way of showcasing the usefulness of the Reidemeister--Schreier method, as well as the rich depth of one-relator group theory in its connections to right-angled Artin groups and generalisations. If the article makes the reader already interested in the Diophantine problem more interested in the subject of one-relator group theory, or vice versa, then it will have fulfilled its primary goal. 

\section{Preliminaries}

\noindent We assume the reader is reader with the basics of group presentations. The textbooks \cite{Magnus1966, Lyndon1977} and the monographs \cite{Adian1966, Baumslag1993} are all excellent places for learning about this subject. We also assume some familiarity with concepts such as hyperbolic groups and algorithmic problems in groups. All groups in this article, if not explicitly stated otherwise, will be assumed to be finitely presented. The notation used for a group $G$ presented with generating set $A$ and defining relations $R$ will be $\pres{Gp}{A}{R}$. The free group on $A$ will be denoted $F_A$. The words \textit{reduced word, cyclically reduced word, cyclic conjugate of a word}, etc. are all used in the usual sense. Conjugation is sometimes written $x^y = yxy^{-1}$, and we use the convention $[x,y] = xyx^{-1}y^{-1}$ for the commutator. 

\subsection{The Reidemeister--Schreier Method}\label{Subsec:Reidemeister}

The Reidemeister-Schreier method dates back to the 1927 papers by Reidemeister \cite{Reidemeister1927} and Schreier \cite{Schreier1927}.\footnote{This latter paper is noteworthy to semigroup theorists, as a key method therein was used already by Hoyer \cite{Hoyer1902} in 1902, who derived the famous Schreier index formula; the method used by Hoyer, remarkably, goes via left (but not right) cancellative semigroups. See \cite[p.48--49]{Chandler1982}.} We only give brief details of this method, borrowing from \cite{Baumslag2008} (an article which served, in part, as direct inspiration of the present article). More details and plenty of examples can be found in e.g. \cite[Chapter~III]{Baumslag1993}. 

Let $A$ be a finite set, and $F_A$ be the free group with basis $A$. Let $G = \pres{Gp}{A}{R_i = 1 \:(i \in I)}$ be a finitely presented group. Let $H$ be subgroup of $G$, generated by the finitely many words $\{ h_1,h_2, \dots, h_k \}$. Let $T$ be a right transversal for $H$ in $G$ which is closed under taking prefixes. That is, we require the following of $T$:
\begin{enumerate}
\item from each coset $Hg$ of $H$ in $G$, there is exactly one element, which we denote by $\overline{g}$, in $T$.
\item if $w \equiv a_1^{\varepsilon_1} a_2^{\varepsilon_2} \cdots a_\ell^{\varepsilon_\ell} \in T$, where $w$ is reduced and $a_i \in A, \varepsilon_i = \pm 1$, then $a_1^{\varepsilon_1} a_2^{\varepsilon_2} \cdots a_k^{\varepsilon_k}$ for every $k \leq \ell$. 
\end{enumerate}
If $T$ satisfies (1) and (2), then we say that $T$ is a \textit{Schreier transversal} for $H$ in $G$, and the element $\overline{g} \in T$ will be called the \textit{Schreier representative} for the coset $Hg$. 

For every $t \in T$ and $a \in A$, we define a set of generators:
\[
s(t, a) = ta \overline{ta}^{-1} \in F_A,
\]
and it is not difficult to prove (as Reidemeister and Schreier did, more or less explicitly), that 
\[
S = \{ s(t,a) \mid t \in T \text{ and } a \in A \}
\]
is a generating set for $H$. Furthermore, we can find a presentation for $H$ in terms of these generators, which we will now do. 

Let $\Sigma$ be a new set of symbols, in one-to-one correspondence with $S$ via $s(t,x) \mapsto \sigma(t,x)$. Then the \textit{rewriting mapping} $\tau \colon F_A \to F_\Sigma$ will be defined as follows.

Let $w = a_1^{\varepsilon_1} a_2^{\varepsilon_2} \cdots a_\ell^{\varepsilon_\ell}$, where $a_i \in A$ and $\varepsilon_i = \pm 1$. Let $w_i$ denote the $i$th prefix of $w$, i.e. the word defined by $w_0 = 1$, and $w_i = a_1^{\varepsilon_1} a_2^{\varepsilon_2} \cdots a_i^{\varepsilon_i}$ for $i>0$. Then $\tau$ is defined by replacing in $w$ each term $a_i^{\varepsilon_i}$ by 
\[
a_i^{\varepsilon_i} \mapsto
\begin{cases*}
\sigma(\overline{w_{i-1}}, a_i) & if $e_i = 1$, \\
\sigma(\overline{w_{i}}, a_i)^{-1} & if $e_i = 1$,
\end{cases*}
\]
Thus $\tau(w)$ will be a word in $F_\Sigma$. Any $F_X$-word representing an element of $H$ can be rewritten by $\tau$. The following useful elementary properties now hold:

\begin{lemma}
Let $T$ be a Schreier transversal for a subgroup $H$ of $G = \pres{Gp}{A}{R_i=1 \: (i \in I)}$. Then for every $w \in F_A$, $t, t' \in T$ and $a, a' \in A$:

\begin{enumerate}
\item $\overline{wx} = \overline{\overline{w}x}$.
\item $\sigma(t,a)=1$ in $F_A$ if and only if $ta \in T$. In particular, $\tau(s(t,a)) = \sigma(t,a)$.
\item If neither $\sigma(t,a)$ nor $\sigma(t',a')$ equal the identity of $F_A$, then $\sigma(t,a) = \sigma(t',a')$ only if 
\[
t = t' \quad \text{and} \quad x = x'.
\]
\end{enumerate}
\end{lemma}

Even if the word $s(t, a)$ is the identity in $F_A$, the symbol $\sigma(t,a)$ is of course never the identity in $F_\Sigma$. Let therefore
\[
\Sigma_1 = \{ \sigma(t,x) \mid s(t,x) = 1 \text{ in $F_X$}\} \subseteq \Sigma.
\]
The Reidemeister--Schreier method gives a presentation for $H$ as follows:
\begin{equation}\label{Eq:RM-method}
H \cong \pres{Gp}{\Sigma - \Sigma_1}{\tau(tR_it^{-1}) = 1 \: (i \in I, t \in T)}
\end{equation}

There is a shortcut to the often rather laborious task of rewriting $\tau(tR_it^{-1}$, which will be useful quite often in the present article. Let $w \in F_A$ be such that $w$ represents an element of $H$. Then, by the above properties, it follows that one can find some Schreier generators $s(t_1,a_1), s(t_2,a_2), \dots, s(t_k,a_k) \in S$ such that 
\[
w = s(t_1,a_1) s(t_2,a_2) \cdots s(t_k,a_k)
\]
where the equality is in $F_A$. Then $\tau(w) = \sigma(t_1,a_1) \sigma(t_2,a_2) \cdots \sigma(t_k,a_k)$. 

Note that if $[G : H] < \infty$, i.e. if $H$ has finite index in $G$, then $|T|$ is finite, so the presentation \eqref{Eq:RM-method} is a finite presentation for $H$. That is, every subgroup of finite index in a finitely presented group is itself finitely presented. Furthermore, given a transversal $T$ for $H$ in $G$, then for any $w \in F_A$ it is straightforward to compute $\overline{w}$. 

Finally, we remark that we will often, for ease of notation, be somewhat sloppy in distinguishing elements of $F_A$ and $F_\Sigma$. For example, if $\tau(s(t,x)) = \sigma(t,x)$, then we may write this as ``let $\sigma(t,x)$ denote $s(t,x)$''. Context will always make our choices clear.

\subsection{Right-angled Artin groups}\label{Subsec:RAAG-intro}

A group is said to be a \textit{right-angled Artin group} if it can be constructed in the following way: let $\Gamma$ be an undirected graph with vertices $x_1, x_2, \dots, x_n$, and edges $E$. Then the right-angled Artin group $A(\Gamma)$ defined by $\Gamma$ is the group given by a presentation with generators $x_1, x_2, \dots, x_n$ and defining relations $[x_i, x_j] = 1$ whenever $(x_i, x_j) \in E$. Thus, for example, if $K_n, E_n$ are the complete resp. the empty graph on $n$ vertices, then $A(K_n) \cong \mathbb{Z}^n$ and $A(E_n) \cong F_n$, the free group on $n$ generators. For various $\Gamma$, the group $A(\Gamma)$ is thus ``intermediately'' or ``partially'' commutative.

Right-angled Artin groups were introduced in 1981 by Baudisch \cite{Baudisch1981} under the name of \textit{semifree} groups. In subsequent years, right-angled Artin groups (or \raags) have seen a great deal of study; their importance seems to derive mainly from the fact that their subgroup structure has been found to be incredibly rich. For instance, the class of \textit{special} groups defined by Haglund \& Wise \cite{Haglund2008} all embed in \raags, and this class of groups is crucial in the recent resolution \cite[\S16.1]{Wise2012} to Baumslag's conjecture that all one-relator groups with torsion are residually finite. Additionally, \raags{} play a r\^ole in Gray's recent proof (see \cite{Gray2020}) that the word problem in one-relator inverse monoids can be undecidable (see also \S\ref{Subsec:SMMP}).

\subsection{Subgroups of one-relator groups}\label{Subsec:OR-groups}

Subgroups of one-relator groups is one of the central topics of this article. A first natural thought may be that \textit{all subgroups of one-relator groups are themselves one-relator groups}, by analogy with the fact that any subgroup of a zero-relator group (i.e. a free group) is again a zero-relator group, which is the Nielsen-Schreier theorem \cite{Nielsen1921, Schreier1927}. However, this is not the case. For example, let $G = \pres{Gp}{a,b}{a^2=1}$. Then $G \cong \langle b \rangle \ast C_2 \cong \mathbb{Z} \ast C_2$. By the Kurosh subgroup theorem, for every $k \geq 0$ the group generated by all $b^i a b^{-i}$ with $0 \leq i \leq k$ is isomorphic to the free product of $k+1$ copies of $C_2$. Clearly this product cannot be defined using fewer than $k+1$ defining relations. Though this example shows that not every one-relator group is \textit{locally} a one-relator group (i.e. every finitely generated subgroup is one-relator), there are some examples of when this does occur. The most prominent example, other than free groups, of locally one-relator groups are surface groups: it is not hard to show via topological methods that every finitely generated subgroup of \eqref{Eq:surface-group} is either free or again a group of type \eqref{Eq:surface-group} (with some genus $g' > 1$). Outside of one-relator subgroups of one-relator groups, we have the following recent result:

\begin{lemma}[Howie \cite{Gray2020}]\label{Lem:Howie'sLemma}
Let $G = \pres{Gp}{A}{w=1}$ be a one-relator group, and let $\Gamma$ be a graph. Then $G$ contains the right-angled Artin group $A(\Gamma)$ as a subgroup only if $\Gamma$ is a forest. 
\end{lemma}

Recall (see e.g. \cite{Magnus1966}) that a one-relator group has torsion if and only if its presentation is of the form $\pres{Gp}{A}{w^n = 1}$ for some $n>1$ with $w$ non-empty and cyclically reduced. If $G$ is a one-relator group with torsion, then the only right-angled Artin groups contained in $G$ are free, as $G$ is hyperbolic (see below). Thus, Lemma~\ref{Lem:Howie'sLemma} has content only in the case of $G$ being torsion-free (in which case it has rather a good deal of content). The abelian subgroups of one-relator groups \cite{Newman1968b} and the solvable subgroups of one-relator groups with torsion \cite{Newman1973} have also been classified. 

Perhaps the most striking theorem about subgroups of one-relator groups is Magnus' \textit{Freiheitssatz}, proved in 1930. This states the following: let $G = \pres{Gp}{a,b,c, \dots}{w=1}$ where $w$ is cyclically reduced, and $a^{\pm 1}$ occurs somewhere in $w$. Then $\{ b, c, \dots \}$ freely generates a subgroup of $G$. Since then, proving truly general theorems about subgroups of one-relator groups has proved rather difficult.\footnote{It is worth mentioning that B. B. Newman proved a strong generalisation of the \textit{Freiheitssatz} for one-relator group with torsion, see \cite[Corollary~2.1.6]{Newman1968}.} The proof of the \textit{Freiheitssatz} is originally via the Reidemeister--Schreier method, and consists (with minor modifications, if necessary) in computing a presentation for the normal subgroup $H$ normally generated by all generators except $a$, i.e. the subgroup generated by $\{ a^i b a^{-i}, a^i c a^{-i}, \dots\}$, of $G$. The relators $\tau(a^i R a^{-i})$ are all shorter than $R$, and the group $H$ essentially has the structure of a direct limit of a chain of amalgamated free products of a fixed one-relator group $G'$, where $N$ is defined by the relator $\tau(R)$. As $|\tau(R)| < |R|$, this gives, as long as one understands the structure of the amalgamated free products (which is not too difficult), an inductive procedure for proving statements about $G$, by passing to the new one-relator group $G'$, then $G''$, eventually reaching a free group.

Given the above, it is natural, given a one-relator group $G= \pres{Gp}{a,b,c,\dots}{w=1}$, to ask about the subgroup of $G$ normally generated by all generators $b, c, \dots$ \textit{and $a^m$} for some $m > 0$. This is clearly a finite index subgroup of $G$, and will always be finitely generated as a group. It turns out that this subgroup -- and variations of it -- often has interesting properties (for example: it is sometimes a right-angled Artin group). We will use such subgroups to prove many of the results in \S\ref{Sec:Moldavanskii-sec} and \S\ref{Sec:Torsion-subgroups}.

We will make some brief remarks about some recent (and some not so recent) theorems about one-relator groups with torsion, which we shall use in \S\ref{Sec:Torsion-subgroups}. Generally speaking, it has turned out that one-relator groups with torsion behave more like free groups, and hence are easier to understand, than their torsion-free counterparts. This includes their subgroup structure. One way in which this is true, other than the aforementioned result on solvable subgroups, comes via hyperbolicity. First, we note a very useful result due to Gersten. 

\begin{theorem}[{Gersten \cite{Gersten1996}}]
Let $H$ be any finitely presented subgroup of a hyperbolic one-relator group. Then $H$ is hyperbolic.
\end{theorem}

We may now combine the above theorem with the following three facts:
\begin{enumerate}
\item One-relator groups with torsion are hyperbolic, by the Spelling Theorem \cite{Newman1968, NybergBrodda2021a};
\item One-relator groups with torsion are coherent, i.e. every finitely generated subgroup is finitely presented, as recently proved \cite{Louder2020}; and
\item The Diophantine problem is decidable in all hyperbolic groups \cite{Dahmani2010}. 
\end{enumerate}
This combination yields the following corollary, which will serve as the starting point for our investigations and results in \S\ref{Sec:Torsion-subgroups}. 

\begin{corollary}\label{Cor:f.g.-ORwt-is-hyperbolic}
Let $H$ be any finitely generated subgroup of a one-relator group with torsion. Then $H$ has decidable Diophantine problem. 
\end{corollary}

The above theorem and its corollary were pointed out to me by H. Wilton. In fact, we will not need the full strength of Corollary~\ref{Cor:f.g.-ORwt-is-hyperbolic} for the results of \S\ref{Sec:Torsion-subgroups}, as we shall only consider finite index subgroups of hyperbolic groups there, which are easily seen to be themselves hyperbolic; however, as part of the general theme of studying subgroups of one-relator groups with torsion to find new torsion-free one-relator groups with decidable Diophantine problem, the above theorem will be crucial. 

\begin{remark}
It is a famous open problem, posed by Baumslag \cite{Baumslag1971}, whether all one-relator groups are coherent. Recently, this problem has shown to be equivalent to a problem concerning the group of units of certain one-relation inverse semigroups, and thereby also to the word problem for one-relation monoids, see \cite{Gray2021}; see also \cite{NybergBrodda2022, NybergBrodda2021}. 
\end{remark}

\section{Warmup -- Torus knot groups}\label{Sec:Warmup-torus-knots}

\noindent We will begin this article by considering a well-studied class of one-relator groups. These are the groups with two generators and, for some $m,n \geq 2$, the single defining relation $a^m b^n=1$. This class of groups is one of the first-studied classes of one-relator groups -- indeed, it is one of the first studied classes of groups in all combinatorial group theory, after free groups. The groups were studied in a combinatorial manner by Schreier \cite{Schreier1924} in 1924, who determined their center and automorphism group. Dehn \cite{Dehn1914}, as part of his investigations of the trefoil knot, had also studied a subclass of this class. Indeed, the link between knots and the groups noted above is quite direct: the fundamental group of the complement of the $(p,q)$-torus knot is isomorphic to the group with the single relation $a^pb^q = 1$. Note, however, that topological considerations force $\gcd(p,q) = 1$, whereas Schreier's combinatorial methods required no such restriction. In recent years, in spite of their seemingly simple nature, have appeared, along with various generalisations, in works in geometric group theory, see e.g. \cite{Niblo2001, Picantin2003, Katayama2017}

\begin{proposition}\label{Prop:Torus-knot-virtually-is-Fn-Z}
Let $m,n \geq 2$, and let $T_{m,n} = \pres{Gp}{a,b}{a^mb^n =1}$. Then $T_{m,n}$ has an index $mn$ characteristic subgroup isomorphic to the right-angled Artin group $\mathbb{Z} \times F_{(m-1)(n-1)}$.
\end{proposition}
\begin{proof}
Let $G = T_{m,n}$, and consider the subgroup of $G$ given by 
\[
H = \langle a^m, [a^i, b^j] \: (i, j \in \mathbb{Z})\rangle
\]
First, $H$ is characteristic in $G$. Indeed, $\langle a^m \rangle = Z(G)$ by easy and classical results \cite{Schreier1924}. Hence
\[
G / Z(G) \cong \pres{Gp}{a,b}{a^mb^n = 1, a^m = 1} \cong C_m \ast C_n,
\]
and taking the quotient of this group by its commutator subgroup, we find $(G/Z(G))^{\text{ab}} \cong C_m \times C_n$. As the commutator subgroup of $G / Z(G)$ is generated by the basic commutators $[a^i, b^j]$, we have that $H$ coincides with this group. Consequently $H$ is a characteristic subgroup of a characteristic subgroup of $G$, so $H$ is characteristic with index $mn$. Alternatively, we may use Schreier's classification of all automorphisms of $G$ (see \cite{Schreier1924}) and note that $H$ is invariant under all of them. 

As $G/H \cong C_m \times C_n$, a Schreier transversal for $H$ in $G$ is given by 
\[
T = \{ a^i b^j \mid 0 \leq i < m, \: 0 \leq j < n\}.
\]
To compute the representative $\overline{a^ib^j \cdot a}$, note that 
\[
(a^i b^j)^{-1} a^ib^ja \equiv [b^j, a^i]a \in Ha \quad \implies \quad a^ib^ja \in Ha^{i+1}b^j,
\]
using the fact that $Ha = aH$. As $a^m \in H$, and by using a similar argument for the action on $b$, we find
\begin{align*}
\overline{a^ib^j \cdot a} &= a^{i+1\bmod{m}} b^j \\
\overline{a^ib^j \cdot b} &= a^{i}b^{j+1\bmod{n}}.
\end{align*}
Consider the generators $s(t,x) = tx \overline{tx}^{-1}$ for $H$, where $t \in T$ and $x \in \{ a, b\}$. It is easy to verify (using the fact that $Hb^ja^i = Ha^ib^j$) that for all $0 \leq i \leq m-1$ and $0 \leq j \leq n-1$:
\[
s(a^ib^j,a) = \begin{cases*} a^i[b^j, a]a^{-i} & if $i \neq m-1$, \\ 
a^{m-1} b^j a b^{-j} & if $i = m-1.$ 
\end{cases*}
\]
We let $\alpha_{i,j} = a^i[b^j, a]a^{-i}$ for $i\neq m-1$, and let $\gamma_j = a^m[a^{-1}, b^j]$. Similarly 
\[
s(a^ib^j,b) = \begin{cases*} 1 & if $j \neq n-1$, \\ 
a^ib^na^{-i} & if $j = n-1$.
\end{cases*}
\]
We let $\beta_i = a^ib^na^{-i}$. We note two special cases: $a^m = \gamma_0$, and $b^n = \beta_0$. 

Let $R = a^mb^n$, the defining relator of $G$. It is now rather straightforward to verify, but somewhat tedious to discover, how to rewrite the defining relations $\tau(a^ib^j Rb^{-j}a^{-i})$ over these generators. Fortunately, the reader need only do the former. There are two types of relations. The first is in the case of $j=0$. Then the relation is particularly simple, as:
\begin{equation}\label{Eq:tau-ai-R-ai}
\tau(a^iRa^{-i}) = \tau(a^i(a^mb^n)a^{-i}) = \tau(a^ma^ib^na^{-i}) = \gamma_0 \beta_i.
\end{equation}
Suppose $j>0$. Let us consider first the case of $i=0$. 
\[
\tau(b^jRb^{-j}) = \tau(b^ja^mb^nb^{-j}) = \tau(b^ja^mb^{-j}b^n).
\]
Thus, if we can find a product of generators $s(t,x)$ which equals $b^j a^m b^{-j}$ in the free group, then we shall be done. The key observation is now the following lemma.

\begin{lemma}\label{Lem:Free-group-key-lemma}
Let $\ell \geq 0$. Then 
\begin{equation}\label{Eq:free-group-key-lemma}
\prod_{k=0}^{\ell}a^k [b^j, a] a^{-k} = [b^j, a^{\ell+1}].
\end{equation}
\end{lemma}
\begin{proof}
The proof is by induction on $\ell$. If $\ell=0$, then both sides are $[b^j, a]$. Suppose \eqref{Eq:free-group-key-lemma} holds for $\ell < \ell_0$, where $\ell_0 > 0$. We shall prove \eqref{Eq:free-group-key-lemma} for $\ell = \ell_0$. We have, using the inductive hypothesis:
\begin{align*}
\prod_{k=0}^{\ell_0} a^k [b^j, a] a^{-k} &= [b^j, a^{\ell_0}] \left( a^{\ell_0} [b^j, a] a^{-\ell_0} \right) \\
&= b^j a^{\ell_0} b^{-j} a^{-\ell_0} a^{\ell_0} b^j a b^{-j} a^{-1} a^{-\ell_0} \\
&= b^j a^{\ell_0+1} b^{-j} a^{-\ell_0-1} = [b^j, a^{\ell+1}].
\end{align*}
This proves the claim.
\end{proof}
Hence, using Lemma~\ref{Lem:Free-group-key-lemma} we find the equality
\begin{align*}
\left( \prod_{k=0}^{m-2} a^k[b^j, a]a^{-k}\right) (a^{m-1} b^jab^{-j}) &= [b^j, a^{m-1}] a^{m-1}b^jab^{-j} \\
&= b^ja^{m-1}b^{-j}a^{-(m-1)}a^{m-1}b^jab^{-j} \\
&= b^j a^{m} b^{-j}.
\end{align*}
Note that the left-hand side is now just a product $\alpha_{0,j} \alpha_{1,j} \cdots \alpha_{m-2,j} \gamma_j$ of Schreier generators. Thus, replacing each such generator with its symbol, we find
\[
\tau(b^jR b^{-j}) = \tau(b^ja^mb^{-j}b^n) = \alpha_{0,j} \alpha_{1,j} \cdots \alpha_{m-2,j} \gamma_j \beta_0.
\]
Now, we may similarly treat the case of $i>0$. Indeed, it is easy to verify, much like above, that 
\begin{equation}\label{Eq:General-conjugation-torus-knot}
\left( \prod_{k=i}^{m-2} a^k[b^j, a]a^{-k}\right) (a^{m-1} b^jab^{-j}) \left( \prod_{k=0}^{i-1} a^k[b^j, a]a^{-k}\right) = a^i b^j (a^m) b^{-j} a^{-i}
\end{equation}
i.e. the added conjugation by $a^i$ simply cyclically permutes the order of the product. We conclude, for all $i$ and $j\neq 0$, that
\begin{align*}
\tau(a^ib^j (a^mb^n)b^{-j}a^{-i}) &= \tau\left( (a^ib^ja^mb^{-j}a^{-i})(a^ib^j b^n b^{-j} a^{-i})\right) \\ 
&= \tau\left( (a^ib^ja^mb^{-j}a^{-i})(a^ib^n a^{-i})\right) \\
&= \alpha_{i,j} \alpha_{i+1,j} \cdots \alpha_{m-2,j} \gamma_j \alpha_{0,j} \cdots \alpha_{i-1,j} \beta_i,
\end{align*}
where we have used \eqref{Eq:General-conjugation-torus-knot}. Thus we find the presentation:
\[
H \cong \pres{Gp}{\alpha_{i,j}, \beta_i, \gamma_j}{\gamma_0\beta_0 = \gamma_0\beta_1 = \cdots \gamma_0\beta_{n-1} = 1,  X_{i,j}}
\]
where, somewhat abusively, the generators $\alpha_{i,j}$ range over $0 \leq i < m-1$ and $0 < j \leq n-1$; the $\beta_i$ range over $0 \leq i \leq m-1$; and $\gamma_j$ over $0 \leq j \leq n-1$ (in all other cases, the generators are equal to $1$); and the relations $X_{i,j}$ for $0 < j \leq n-1$ are defined as follows (with added spacing for readability):
\[
\left\{
\begin{matrix}
\alpha_{0,j} & \alpha_{1,j} & \cdots & \alpha_{m-3,j} &\alpha_{m-2, j} & \gamma_j & \beta_0  = 1\\
\alpha_{1,j} & \alpha_{2,j} & \cdots & \alpha_{m-2,j} & \gamma_j & \alpha_{0,j} & \beta_1 = 1 \\
 \vdots  & \vdots  &  \vdots     &   \vdots    &    \vdots     &       \vdots       & \vdots \\
 \alpha_{m-2,j} & \gamma_j & \alpha_{0,j} & \cdots & \cdots & \alpha_{m-3,j} & \beta_{m-2} = 1 \\
 \gamma_j & \alpha_{0,j} & \cdots & \cdots  & \cdots & \alpha_{m-2,j} & \beta_{m-1} = 1.
\end{matrix}\right.
\]
That is, the $i$th row of $X_{i,j}$ (as presented above) is just the relation $\tau(a^ib^jRb^{-j}a^{-i})$. Analogously, the relations $\gamma_0\beta_i$ are just the relations \eqref{Eq:tau-ai-R-ai}. Of course, in using e.g. the notation $\alpha_{m-3,j}$ even when $m<3$ the meaning is still clear; for example, the $m$th row in the above relations is simply a product of $\gamma_j$ by $\alpha_{0,j} \cdots \alpha_{m-2,j} \beta_{m-1}$, which is well-defined as $m, n \geq 2$ by assumption. 

We now simplify the presentation. The relations $\gamma_0\beta_i = 1$ for $0 \leq i \leq n-1$ makes this simple. The generator $\gamma_0$ may be removed immediately, yielding that $\beta_i = \beta_{i'}$ for all $0 \leq i, i' \leq n-1$. We keep $\beta_0$, replacing all other $\beta_i$'s by it, yielding a new presentation (call this $H'$). 

Now, using the first row of $X_{i,j}$, we can add the relation $\alpha_{0,j} = (\alpha_{1,j} \cdots \alpha_{m-2,j} \gamma_j \beta_0)^{-1}$ to $H'$, and hence remove the generator $\alpha_{0,j}$ by a Tietze transformation. By considering the $i$th row of $X_{i,j}$, we first find the relation 
\[
\gamma_j (\alpha_{1,j} \cdots \alpha_{m-2,j} \gamma_j \beta_0)^{-1} \alpha_{1,j} \cdots \alpha_{m-2,j} \beta_{0} =  1
\]
which, upon cancelling the $\alpha$'s, simply becomes $\gamma_j \beta_0^{-1} \gamma_j^{-1} \beta_0 = 1$, or equivalently $[\beta_0, \gamma_j] = 1$. Going through the $(i-1)$st row, we find 
\[
 \alpha_{m-2,j} \gamma_j (\alpha_{1,j} \cdots \alpha_{m-2,j} \gamma_j \beta_0)^{-1} \alpha_{1,j}\cdots\alpha_{m-3,j} \beta_{0} = 1.
\]
Cancelling, we find
\[
\alpha_{m-2,j} \gamma_j \beta_0^{-1} \gamma_j^{-1} \alpha_{m-2,j}^{-1} \beta_{0} = 1.
\]
Using the fact that $\beta_0$ and $\gamma_j$ commute, we can hence cancel the $\gamma_j$'s and find
\[
\alpha_{m-2,j} \beta_0^{-1} \alpha_{m-2,j} \beta_0 = 1
\]
or equivalently $[\beta_0, \alpha_{m-2,j}] = 1$. The $(i-2)$nd row, using the fact that $[\beta_0, \gamma_j] = [\beta_0, \alpha_{m-2,j}] = 1$, can now be simplified in an entirely analogous way to yield $[\beta_0, \alpha_{m-3,j}] = 1$, etc. (note that $\gamma_j$ is analogous to ``$\alpha_{m-1,j}$''). Thus we find that the relators $X_{i,j}$ can all be simplified to the following system, for all $0 < j \leq n-1$:
\begin{align}\label{Eq:the-beta0-commute-with-alphas}
[\beta_0, \gamma_j] = [\beta_0, \alpha_{m-2,j}] = [\beta_0, \alpha_{m-3,j}] = \cdots =  [\beta_0, \alpha_{1,j}] = 1.
\end{align}
Now \eqref{Eq:the-beta0-commute-with-alphas} consists of $m-1$ defining relations, and each $1 \leq j \leq n-1$ gives rise to one (distinct) system \eqref{Eq:the-beta0-commute-with-alphas}. Thus $H$ is defined by a presentation which consists of the generators
\begin{align*}
&\alpha_{i,j}, \quad (1 \leq i \leq m-2, \: 1 \leq j \leq n-1) \\
&\gamma_j, \quad\:\:\: (1 \leq j \leq n-1) \\
\textnormal{and } &\beta_0.
\end{align*}
and the $(m-1)(n-1)$ relations given by \eqref{Eq:the-beta0-commute-with-alphas}. Note that $H$ has $(m-2)(n-1) + (n-1) = (m-1)(n-1)$ generators together with $\beta_0$, all of which commute with $\beta_0$. Thus, relabelling the generators, we find
\[
H \cong \pres{Gp}{\alpha, \beta_1, \dots, \beta_{(m-1)(n-1)}}{[\alpha, \beta_i] = 1 \: (\forall i)} \cong  \mathbb{Z} \times F_{(m-1)(n-1)}.
\]
As $H$ has index $mn$ in $G$, the claim is proved. 
\end{proof}

Now, for any $m,n \geq 2$, it follows from Proposition~\ref{Prop:Torus-knot-virtually-is-Fn-Z} that $T_{m,n}$ is virtually a direct product of the hyperbolic group $F_{(m-1)(n-1)}$ by the abelian group $\mathbb{Z}$. Thus, by the recent result due to Ciobanu, Holt \& Rees \cite{Ciobanu2020}, it follows that $T_{m,n}$ has decidable Diophantine problem in this case. When either of $m$ or $n$ is $0$ or $1$, then $T_{m,n}$ is virtually free, so the statement is also true in this case by \cite{Dahmani2010}. We conclude:

\begin{theorem}\label{Thm:warmup-am-bn-diophantine}
For every $m, n \geq 0$ the group $\pres{Gp}{a,b}{a^mb^n=1}$ has decidable Diophantine problem. In particular, the Diophantine problem is decidable in the fundamental group of the complement of any torus knot.
\end{theorem}

Note that the case of the torus knot group is when $\gcd(m,n) = 1$. In this case, Proposition~\ref{Prop:Torus-knot-virtually-is-Fn-Z} has recently been proved by Katayama \cite[Theorem~1.5]{Katayama2017} by geometric means (namely, via $3$-manifold theory). In fact, if $\gcd(m,n) = 1$, then Katayama proves that $mn$ is the least index $k$ such that $T_{m,n}$ has an index $k$ subgroup which embeds in a \raag{}. When $\gcd(m,n) \neq 1$, the corresponding statement is not true -- for example, when $m=n$, the situation is significantly simpler; it is not too difficult to show (using a significantly shorter argument than for Proposition~\ref{Prop:Torus-knot-virtually-is-Fn-Z}) that $T_{m,m}$ has an index $m$ normal subgroup isomorphic to $F_m \times \mathbb{Z}$. Of course, this result, which may be well-known, generalises the fact that the infinite dihedral group $D_\infty$ is virtually abelian. When $\gcd(m,n) \neq 1$ (i.e. in the case not treated by Katayama), it would be interesting to know what the least index $k$ such that $T_{m,n}$ has an index $k$ subgroup either isomorphic to or embedding in a \raag{}. Our Proposition~\ref{Prop:Torus-knot-virtually-is-Fn-Z} provides an upper bound on $k$ as $mn$, which the example of $T_{m,m}$ shows is not sharp when $\gcd(m,n) \neq 1$.

\section{Moldavanskii groups and Baumslag-Solitar star-groups}\label{Sec:Moldavanskii-sec}

\subsection{Right-angled Baumslag-Solitar-Artin groups}\label{Subsec:RABSAG-intro}

Let $\Gamma = (V,E)$ be a finite, directed, edge-labelled graph. A (directed) edge $e$ from $v_1$ to $v_2$ (with $v_1, v_2 \in V$) will be labelled by a non-zero integer $m \in \mathbb{Z}$. We will write this as $e = (v_1, m, v_2)$, or graphically as $e \colon v_1 \xrightarrow{m} v_2$. We do not permit loops, or multiple edges with the same start- and endpoint. For every pair of vertices $v_i, v_j \in V$, if there is an edge $v_i \xrightarrow{m} v_j$ then we set $m_{i,j} = m$.  Otherwise, we set $m_{i,j} = 0$ (by convention). The matrix whose $(i,j)$th entry is $m_{i,j}$ is the \textit{(weighted) adjacency matrix of $\Gamma$}, and is denoted $M(\Gamma)$. From $\Gamma$ we will construct a group $B(\Gamma)$ as follows:
\begin{equation}\label{Def:RABSAG}
B(\Gamma) = \pres{Gp}{V}{v_i v_j^{m_{i,j}} = v_j^{m_{i,j}} v_i \quad \text{for all } (v_i, m_{i,j}, v_j) \in E }
\end{equation}

The group $B(\Gamma)$ is called a \textit{right-angled Baumslag-Solitar-Artin group}, or \rabsag{} for short. We will often speak of the ``underlying undirected graph'' of a graph $\Gamma$ as above; this is simply the graph $\Gamma$ without any direction or labels on the edges. We give some examples in Figure~\ref{Fig:RABSAG-examples}. By convention, any edge $(v_i, m_{i,j}, v_j)$ with $m_{i,j} = \pm 1$, i.e. giving rise to a relation $[v_i, v_j] = 1$, will be drawn as a single undirected edge. 

\begin{figure}[!h]
\centering
\begin{tikzpicture}[>=stealth',thick,scale=0.8,el/.style = {inner sep=2pt, align=left, sloped}]%
                        \node (v0)[label=above:$v_0$][circle, draw, fill=black!50,
                        inner sep=0pt, minimum width=8pt] at (-4,0) {};
                        \node (v1)[label=above:$v_1$][circle, draw, fill=black!50,
                        inner sep=0pt, minimum width=8pt] at (-2,0) {};
                        \node (v2)[label=below:$v_2$][circle, draw, fill=black!50,
                        inner sep=0pt, minimum width=8pt] at (-2,-2) {};
                        \node (v3)[label=below:$v_3$][circle, draw, fill=black!50,
                        inner sep=0pt, minimum width=8pt] at (-4,-2) {};
\path[->] 
    (v0)  edge node[above]{$2$}         (v1)
    (v2)  edge node[below]{$-3$}     (v3)
    (v3)  edge node[left]{$-2$}    (v0);
\path[-]
	(v1)  edge node[right]{}         (v2);

                        \node (w0)[label=below:$w_0$][circle, draw, fill=black!50,
                        inner sep=0pt, minimum width=8pt] at (2,-1) {};
                        \node (w1)[label=below:$w_1$][circle, draw, fill=black!50,
                        inner sep=0pt, minimum width=8pt] at (0,-1) {};
                        \node (w2)[label=below:$w_2$][circle, draw, fill=black!50,
                        inner sep=0pt, minimum width=8pt] at (4,-1) {};
                        
\path[->] 
    (w0)  edge node[above]{$2$}         (w1);
\path[-]
	(w2)  edge node[right]{}         (w0);
	
	\node (lab) at (0, -4) {$B(\Gamma_0) = \pres{Gp}{v_0, v_1, v_2, v_3}{ [v_0, v_1^2]  = [v_1, v_2] = [v_2, v_3^{-3}] = [v_3, v_0^{-2}] = 1}$};
		\node (lab) at (0, -4.75) {$B(\Gamma_1) = \pres{Gp}{w_0, w_1, w_2}{ [w_0, w_2]  = [w_0, w_1^2] = 1}$};
\node (lab) at (-3, -3) {$\Gamma_0$};
\node (lab) at (2, -3) {$\Gamma_1$};

\end{tikzpicture}
\caption{Two \rabsags{} $B(\Gamma_0)$ and $B(\Gamma_1)$ with corresponding graphs $\Gamma_0$ and $\Gamma_1$.}
\label{Fig:RABSAG-examples}
\end{figure}

We note some preliminary properties of $B(\Gamma)$. We let $n = |V|$.

\begin{enumerate}
\item If $\Gamma$ is non-empty, then $B(\Gamma)$ is infinite. Indeed, its abelianisation is clearly $\mathbb{Z}^n$. 
\item If $M(\Gamma)$ is the zero matrix, then $B(\Gamma)$ is a free group of rank $n$. 
\item If $M(\Gamma)$ is a binary matrix (i.e. $m_{i,j} \in \{ 0, 1\}$), then $B(\Gamma) = A(\Gamma)$, the right-angled Artin group on the (undirected) graph $\Gamma$. 
\item If $\Gamma$ is a graph on two vertices $u, v$, with one edge $(u, m, v)$, then $B(\Gamma) \cong \BS(m,m)$. 
\end{enumerate}

We let $B_m(\Gamma)$ denote the monoid with the same presentation as \eqref{Def:RABSAG}. The following follows immediately from Adian's classical theory and a result due to Sarkisian.

\begin{proposition}\label{Prop:Basics_of_BGamma}
If $\Gamma$ is a finite forest (as an undirected graph), then:
\begin{enumerate}
\item $B_m(\Gamma)$ embeds in $B(\Gamma)$;
\item the word problem is decidable in $\Gamma$. 
\end{enumerate}
\end{proposition}
\begin{proof}
As every relation in $B(\Gamma)$ is of the form $v_i v_j^{m_{i,j}} = v_j^{m_{i,j}} v_i$, it follows that the left and right graphs (in the sense of Adian \cite{Adian1966}) of $B_m(\Gamma)$ are both isomorphic as undirected graphs to $\Gamma$. In particular, $B_m(\Gamma)$ is a cycle-free monoid, and so embeds in $B(\Gamma)$ (via the identity map). 

As every relation $u=v$ in $B_m(\Gamma)$ satisfies $|u|=|v|$, the word problem as well as the left and right divisibility problems in $B_m(\Gamma)$ are almost trivial to solve. As $B_m(\Gamma)$ is cycle-free, the decidability of these problems suffices for the word problem in $B(\Gamma)$ to be decidable by \cite[Theorem~3]{Sarkisian1979}. 
\end{proof}

In view of Proposition~\ref{Prop:Basics_of_BGamma}, the following two problems naturally present themselves:

\begin{question}
Let $\Gamma$ be any finite, directed edge-labelled graph as above. Does $B_m(\Gamma)$ embed in $B(\Gamma)$? Is the word problem decidable in $B(\Gamma)$? Is $B(\Gamma)$ automatic?
\end{question}

All three questions have affirmative answers in the case of right-angled Artin groups, i.e. when $M(\Gamma)$ is a binary matrix: Paris \cite{Paris2002} proved that every right-angled Artin monoid (defined analogously) embeds in its corresponding right-angled Artin group (this can also be deduced in a simpler manner via rewriting systems \cite{Chouraqui2009}); Green \cite{Green1990} solved the word problem via normal forms; and Hermiller \& Meier \cite{Hermiller1995} constructed finite complete rewriting systems for right-angled Artin groups which additionally lead to an automatic structure.

Particularly relevant to the present article is the following question.

\begin{question}\label{Q:DP-in-rabsags?}
Is the Diophantine problem decidable in all \rabsags?
\end{question}

The Diophantine problem is decidable in all right-angled Artin groups \cite{Diekert2006}. A positive answer to the following question would yield a positive answer to \eqref{Q:DP-in-rabsags?}, by virtue of the result for \raags{} and the result by Levine \cite{Levine2021} that decidability of the Diophantine problem is inherited by taking finite extensions, 

\begin{question}
Is every \rabsag{} virtually a \raag{}?
\end{question}

The answer appears affirmative at least in the case that $\Gamma$ has only a single directed edge with a label $p>1$, in which case an index $p$ \raag{} subgroup can often be found in $B(\Gamma)$. 

One may at first glance expect there to be no substantial added differences when working with \rabsags{} rather than \raags{}. However, as we shall see in \S\ref{Subsec:SMMP}, there is a somewhat marked increase in difficulty for some decision problems. We will first prove some embeddability properties between \raags{} and \rabsags{} to demonstrate this difficulty, and connect it with the submonoid membership problem.

\subsection{The submonoid membership problem}\label{Subsec:SMMP}

We borrow some terminology from \cite{Lohrey2008}. We say that an undirected graph $\Gamma$ is $C_4$-\textit{free} (or \textit{chordal}) if it does not contain any induced subgraph isomorphic to $C_4$, the cycle on four vertices, i.e. if every cycle on four vertices in $\Gamma$ has a chord. We say that $\Gamma$ is \textit{$P_4$-free} if it does not contain any induced subgraph isomorphic to $P_4$, the path graph on four vertices. Finally, $\Gamma$ is a \textit{transitive forest} if it has no induced subgraph isomorphic to either $C_4$ or $P_4$. Note that every connected component of a transitive forest contains a vertex which is adjacent to every vertex in the component. In particular, if $\Gamma$ is a connected transitive forest, then $A(\Gamma)$ has non-trivial center. 

\begin{theorem}[Lohrey \& Steinberg, 2008]\label{Thm:Lohrey-Steinberg}
Let $G = A(\Gamma)$ be a right-angled Artin group. Then:
\begin{enumerate}
\item The subgroup membership problem in $G$ is decidable only if $\Gamma$ is $C_4$-free;
\item The submonoid membership problem in $G$ is decidable if and only if $\Gamma$ is $P_4$-free;
\item The submonoid membership problem in $G$ is recursively equivalent to the rational subset membership problem in $G$.
\end{enumerate}
\end{theorem}

\begin{remark}
It is an open problem whether there exists some group $G$ in which the rational subset membership problem is undecidable but the submonoid membership problem is decidable. Indeed, such a group exists if and only if decidability of the submonoid membership problem is \textbf{not} in general preserved under taking free products \cite{Lohrey2008}. It is also an open problem whether part (1) of Theorem~\ref{Thm:Lohrey-Steinberg} is sharp; for instance, it is unknown whether the subgroup membership problem is undecidable in $A(C_5)$ (and $A(C_5) \leq A(C_k) \leq A(C_5)$ for every $k \geq 5$ by \cite[Theorem~11]{Kim2013}). 
\end{remark}

Droms \cite{Droms1987} proved that if $\Gamma$ is a transitive forest, then every finitely generated subgroup of $A(\Gamma)$ is isomorphic to $A(\Gamma')$ for some transitive forest $\Gamma'$. We will now show (Propositions~\ref{Prop:BGamma_1-contains-P4} and \ref{Prop:Chordal-RABSAG-contains-C4}) that the analogous statement for right-angled Baumslag-Solitar-Artin groups fails remarkably easily.

\begin{proposition}\label{Prop:BGamma_1-contains-P4}
Let $G = B(\Gamma_1)$, where $\Gamma_1$ is the graph 
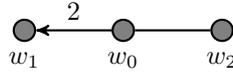
\begin{center}
\begin{tikzpicture}[>=stealth',thick,scale=1.3,el/.style = {inner sep=2pt, align=left, sloped}]%

                        \node (w0)[label=below:$w_0$][circle, draw, fill=black!50,
                        inner sep=0pt, minimum width=8pt] at (0,0) {};
                        \node (w1)[label=below:$w_1$][circle, draw, fill=black!50,
                        inner sep=0pt, minimum width=8pt] at (-1,0) {};
                        \node (w2)[label=below:$w_2$][circle, draw, fill=black!50,
                        inner sep=0pt, minimum width=8pt] at (1,0) {};
                        
\path[->] 
    (w0)  edge node[above]{$2$}         (w1);
\path[-]
	(w2)  edge node[above]{}         (w0);
\end{tikzpicture}
\end{center}
Then $A(P_5) \leq G$ as an index $2$ subgroup. In particular, the submonoid membership problem is undecidable in $G$.
\end{proposition}
\begin{proof}
Let $H$ be the subgroup of $G$ generated by 
\[
\{ w_1^2, w_1^i w_0 w_1^{-i}, w_1^i w_2 w_1^{-i} \: (i \in \mathbb{Z}) \}.
\] 
This is obviously a normal subgroup. A presentation for $G/H$ is found by a simple Tietze transformation to be $\pres{Gp}{w_1}{w_1^2 = 1}$, so $H$ has index $2$ in $G$. We will show that $H \cong A(P_5)$. A Schreier transversal for $G/H$ is $T = \{ 1, w_1 \}$. Let $s_{t,i} = tw_i \overline{tw_i}^{-1}$ for $t \in T, i \in \{ 0,1,2\}$. Then the collection $S$ of all such $s_{t,i}$ generates $H$. Now as $w_0, w_2 \in H$, we find
\begin{equation*}
\overbrace{s_{1,0} = w_0}^{\beta_1}, \qquad s_{1,1} = 1, \qquad \overbrace{s_{1,2} = w_2}^{\beta_0}
\end{equation*}
\begin{equation*}
\underbrace{s_{w_1,0} = w_1 w_0 w_1^{-1}}_{\beta_3}, \qquad \underbrace{s_{w_1,1} = w_1^2}_{\beta_2}, \qquad \underbrace{s_{w_1,2} = w_1 w_2 w_1^{-1}}_{\beta_4}.
\end{equation*}
A defining set of relators for $H$ is now, by the Reidemeister-Schreier method, the set of relators $tRt^{-1}$, where $t \in T$ and $R$ is a relator for $G$, rewritten using a rewriting function $\tau$ from the alphabet $\{ w_0, w_1, w_2\}$ to the alphabet $\{ \beta_0, \beta_1, \dots, \beta_4\}$. Let $R_1 = [w_0, w_1^2]$, and $R_2 = [w_0,w_2]$. Then the defining relators for $H$ become:
\begin{align*}
\tau(&R_1) = \tau(w_0 w_1^2 w_0^{-1} w_1^{-2}) \equiv \tau \left((w_0)(w_1^2)(w_0)^{-1}(w_1^{2})^{-1}\right) = \beta_1 \beta_2 \beta_1^{-1} \beta_2^{-1}, \\
\tau(w_1 &R_1 w_1^{-1}) = \tau(w_1(w_0 w_1^2 w_0^{-1} w_1^{-2})w_1^{-1}) \equiv \tau \left( (w_0^{w_1}) (w_1^2) (w_0^{w_1})^{-1} (w_1^2)^{-1} \right) = \beta_3 \beta_2 \beta_3^{-1} \beta_2^{-1}, \\
\tau(&R_2) = \tau(w_0 w_2 w_0^{-1} w_2^{-1}) = \beta_1 \beta_0 \beta_1^{-1} \beta_0^{-1}, \\
\tau(w_1&R_2 w_1^{-1}) = \tau(w_1 w_0 w_2 w_0^{-1} w_2^{-1} w_1^{-1}) \equiv \tau\left( w_0^{w_1} w_2^{w_1} (w_0^{-1})^{w_1} (w_2^{-1})^{w_1} \right) = \beta_3 \beta_4 \beta_3^{-1} \beta_4^{-1}.
\end{align*}
Hence a presentation for $H$ is given by 
\[
H \cong \pres{Gp}{\beta_0, \dots, \beta_4}{[\beta_i, \beta_{i+1}] = 1 \:\: (0 \leq i \leq 3} \cong A(P_5),
\]
i.e. $H$ is isomorphic to the \raag{} with underlying graph a path on the five vertices $\beta_0, \dots, \beta_4$. 
\end{proof}

More generally, we denote by $P_{3,k}$ the following graph with three vertices:
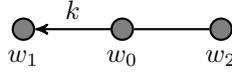
\begin{center}
\begin{tikzpicture}[>=stealth',thick,scale=1.3,el/.style = {inner sep=2pt, align=left, sloped}]%

                        \node (w0)[label=below:$w_0$][circle, draw, fill=black!50,
                        inner sep=0pt, minimum width=8pt] at (0,0) {};
                        \node (w1)[label=below:$w_1$][circle, draw, fill=black!50,
                        inner sep=0pt, minimum width=8pt] at (-1,0) {};
                        \node (w2)[label=below:$w_2$][circle, draw, fill=black!50,
                        inner sep=0pt, minimum width=8pt] at (1,0) {};
                        
\path[->] 
    (w0)  edge node[above]{$k$}         (w1);
\path[-]
	(w2)  edge node[above]{}         (w0);
\end{tikzpicture}
\end{center}
The graph $\Gamma_1$ in Proposition~\ref{Prop:BGamma_1-contains-P4} is then $P_{3,2}$. By the same reasoning as in Proposition~\ref{Prop:BGamma_1-contains-P4}, we find that $B(P_{3,k})$ contains an index $k$ subgroup isomorphic to $A(\Gamma_k)$, where $\Gamma_k$ is an undirected graph obtained as follows: take $k$ copies of the (undirected) path on three vertices, and root each such graph at one of the leaves. The graph $\Gamma_k$ is then the ``wedge sum'' of the graphs in the roots, i.e. the union of the $k$ graphs with the roots all identified, having $2k+1$ vertices in total. Clearly, $\Gamma_k$ contains an induced subgraph isomorphic to $P_4$ for all $k \geq 2$, so $A(P_4) \leq A(\Gamma_k) \leq B(P_{3,k})$ for all such $k$. We leave the details to the interested reader. We conclude by Theorem~\ref{Thm:Lohrey-Steinberg} the following: 

\begin{lemma}\label{Lem:SMMP-in-BP3k-and-induced}
The submonoid membership problem is undecidable in $B(P_{3,k})$ when $k \geq 2$. In particular, let $\Gamma$ be any directed graph containing $P_{3,k}$ as an induced subgraph. Then $B(P_{3,k}) \leq B(\Gamma)$, and the submonoid membership problem is undecidable in $B(\Gamma)$. 
\end{lemma}

Here we have used the obvious fact that if $\Gamma$ is an induced subgraph of $\Gamma'$, then $B(\Gamma) \leq B(\Gamma')$ via the map induced from the inclusion of the vertices of $\Gamma'$ into $\Gamma$. We shall use Lemma~\ref{Lem:SMMP-in-BP3k-and-induced} presently, by giving another family of \rabsags{} with undecidable submonoid membership problem. For $k,\ell \in \mathbb{Z}$, we denote by $S_{k,\ell}$ \textit{the $\ell$-regular Baumslag-Solitar-Artin star with $k$ arms} given in Figure~\ref{Fig:BS-star}. We let $S^{(b)}_{k,\ell}$ be the \textit{bristled} graph of $S_{k,\ell}$, also presented in Figure~\ref{Fig:BS-star}.
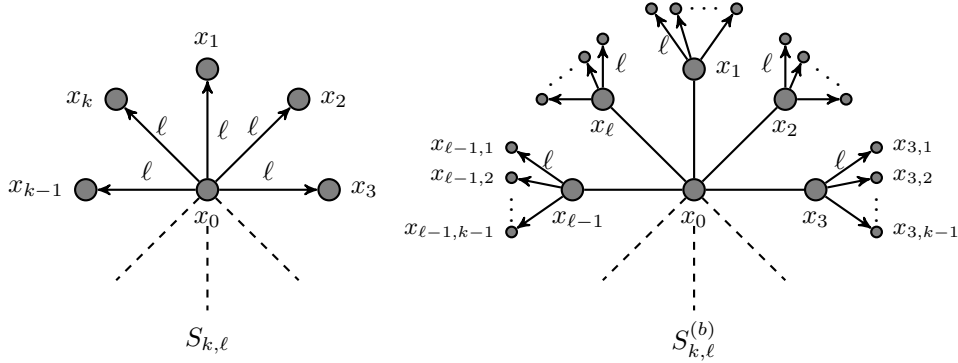
\begin{figure}[h]
\centering
\begin{tikzpicture}[>=stealth',thick,scale=0.8,el/.style = {inner sep=2pt, align=left, sloped}]%

\tikzmath{\t = 1;}

\node (w0)[label=below:$x_0$][circle, draw, fill=black!50,                        inner sep=0pt, minimum width=8pt] at (0,0) {};
\node (w1)[label=above:$x_1$][circle, draw, fill=black!50,
inner sep=0pt, minimum width=8pt] at (0,2) {};
\node (w2)[label=right:$x_2$][circle, draw, fill=black!50,
inner sep=0pt, minimum width=8pt] at (1.5,1.5) {};
\node (w3)[label=right:$x_3$][circle, draw, fill=black!50,
inner sep=0pt, minimum width=8pt] at (2,0) {};
\node (w4)[label=left:$x_{k-1}$][circle, draw, fill=black!50,
inner sep=0pt, minimum width=8pt] at (-2,0) {};
\node (w5)[label=left:$x_{k}$][circle, draw, fill=black!50,
inner sep=0pt, minimum width=8pt] at (-1.5,1.5) {};

\path[->]
	(w0)  edge node[right]{$\ell$}         (w1)
	(w0)  edge node[above]{$\ell$}         (w2)
	(w0)  edge node[above]{$\ell$}         (w3)
	(w0)  edge node[above]{$\ell$}         (w4)
    (w0)  edge node[above]{$\ell$}         (w5);
\draw [dashed] (w0) -- (1.5,-1.5);
\draw [dashed] (w0) -- (-1.5,-1.5);
\draw [dashed] (0,-0.6) -- (0,-2);
\node (lab) at (0,-2.5) {$S_{k,\ell}$};

\node (lab) at (\t+7,-2.5) {$S^{(b)}_{k,\ell}$};
\node (v0)[label=below:$x_0$][circle, draw, fill=black!50,
inner sep=0pt, minimum width=8pt] at (\t+7,0) {};
\node (v1)[label=right:$x_1$][circle, draw, fill=black!50,
inner sep=0pt, minimum width=8pt] at (\t+7,2) {};
\node (v2)[label=below:$x_2$][circle, draw, fill=black!50,
inner sep=0pt, minimum width=8pt] at (\t+8.5,1.5) {};
\node (v3)[label=below:$x_3$][circle, draw, fill=black!50,
inner sep=0pt, minimum width=8pt] at (\t+9,0) {};
\node (v4)[label=below:$\:\:x_{\ell-1}$][circle, draw, fill=black!50,
inner sep=0pt, minimum width=8pt] at (\t+5,0) {};
\node (v5)[label=below:$x_{\ell}$][circle, draw, fill=black!50,
inner sep=0pt, minimum width=8pt] at (\t+5.5,1.5) {};

\node (x)[label=left:{\small $x_{\ell-1,1}$}][circle, draw, fill=black!50, inner sep=0pt, minimum width=4pt] at (\t+4,0.7) {};
\node (y)[label=left:{\small $x_{\ell-1,2}$}][circle, draw, fill=black!50, inner sep=0pt, minimum width=4pt] at (\t+4,0.2) {};
\node (z)[label=left:{\small $x_{\ell-1,k-1}$}][circle, draw, fill=black!50, inner sep=0pt, minimum width=4pt] at (\t+4,-0.7) {};
\draw [->] (v4) -- (x); \draw [->] (v4) -- (y); \draw [->] (v4) -- (z);
\node (lab) at (\t+4,-0.2) {$\vdots$}; \node (lab) at (\t+4.6,0.5) {$\ell$};

\node (x)[label=right:{\small $x_{3,1}$}][circle, draw, fill=black!50, inner sep=0pt, minimum width=4pt] at (\t+10,0.7) {};
\node (y)[label=right:{\small $x_{3,2}$}][circle, draw, fill=black!50, inner sep=0pt, minimum width=4pt] at (\t+10,0.2) {};
\node (z)[label=right:{\small $x_{3,k-1}$}][circle, draw, fill=black!50, inner sep=0pt, minimum width=4pt] at (\t+10,-0.7) {};
\draw [->] (v3) -- (x); \draw [->] (v3) -- (y); \draw [->] (v3) -- (z);
\node (lab) at (\t+10,-0.2) {$\vdots$}; \node (lab) at (\t+9.4,0.5) {$\ell$};

\node (x)[][circle, draw, fill=black!50, inner sep=0pt, minimum width=4pt] at (\t+5.5,2.5) {};
\node (y)[][circle, draw, fill=black!50, inner sep=0pt, minimum width=4pt] at (\t+5.2,2.2) {};
\node (z)[][circle, draw, fill=black!50, inner sep=0pt, minimum width=4pt] at (\t+4.5,1.5) {};
\draw [->] (v5) -- (x); \draw [->] (v5) -- (y); \draw [->] (v5) -- (z);
\node (lab) at (\t+4.8,2) {$\iddots$}; \node (lab) at (\t+5.8,2.1) {$\ell$};

\node (x)[][circle, draw, fill=black!50, inner sep=0pt, minimum width=4pt] at (\t+8.5,2.5) {};
\node (y)[][circle, draw, fill=black!50, inner sep=0pt, minimum width=4pt] at (\t+8.8,2.2) {};
\node (z)[][circle, draw, fill=black!50, inner sep=0pt, minimum width=4pt] at (\t+9.5,1.5) {};
\draw [->] (v2) -- (x); \draw [->] (v2) -- (y); \draw [->] (v2) -- (z);
\node (lab) at (\t+9.2,2) {$\ddots$}; \node (lab) at (\t+8.2,2.1) {$\ell$};

\node (x)[][circle, draw, fill=black!50, inner sep=0pt, minimum width=4pt] at (\t+6.3,3) {};
\node (y)[][circle, draw, fill=black!50, inner sep=0pt, minimum width=4pt] at (\t+6.7,3) {};
\node (z)[][circle, draw, fill=black!50, inner sep=0pt, minimum width=4pt] at (\t+7.7,3) {};
\draw [->] (v1) -- (x); \draw [->] (v1) -- (y); \draw [->] (v1) -- (z);
\node (lab) at (\t+7.2,3) {$\cdots$}; \node (lab) at (\t+6.5,2.4) {$\ell$};

\path[-]
	(v0)  edge node[right]{}         (v1)
	(v0)  edge node[above]{}         (v2)
	(v0)  edge node[above]{}         (v3)
	(v0)  edge node[above]{}         (v4)
    (v0)  edge node[above]{}         (v5);
\draw [dashed] (v0) -- (\t+8.5,-1.5);
\draw [dashed] (v0) -- (\t+5.5,-1.5);
\draw [dashed] (\t+7,-0.6) -- (\t+7,-2);

\end{tikzpicture}
\caption{The graph $S_{k,\ell}$ and its bristled graph $S^{(b)}_{k,\ell}$. The group $B(S_{k,\ell})$ is generated by the $x_i$ with the $k$ defining relations $[x_0, x_i^\ell] = 1$ for $1 \leq i \leq k$. The group $B(S^{(b)}_{k,\ell})$ is generated by the $x_i$ and $x_{i,j}$, with the defining relations $[x_0, x_i] = 1$ and $[x_i, x_{i,j}^\ell] = 1$ for all $1 \leq i \leq k$ and $1 \leq j \leq \ell$. In Proposition~\ref{Prop:Star-BS-contains-P3ell}, we will prove that $B(S_{k,\ell})$ has an index $\ell$ normal subgroup isomorphic to $B(S^{(b)}_{k,\ell})$.}
\label{Fig:BS-star}
\end{figure}

We will show that $B(S_{k,\ell})$ contains a subgroup isomorphic to $B(P_{3,\ell})$  whenever $k, \ell \geq 2$. Note that $S^{(b)}_{k,\ell}$ contains an induced subgraph isomorphic to $P_{3,\ell}$ -- indeed, take the subgraph induced on $\{ x_0, x_i, x_{i,j} \}$ for any $i, j \geq 1$. Hence $B(S^{(b)}_{k,\ell})$ contains a subgroup isomorphic to $B(P_{3,\ell})$ (namely, the group generated by $x_0, x_i$, and $x_{i,j}$).

\begin{proposition}\label{Prop:Star-BS-contains-P3ell}
Let $G = B(S_{k,\ell})$, and let $H = B(S^{(b)}_{k,\ell})$, where $k, \ell \geq 2$. Then $H$ is isomorphic to an index $\ell$ normal subgroup of $G$. In particular, $G$ contains a subgroup isomorphic to $P_{3,\ell}$.
\end{proposition}
\begin{proof}
We distinguish the generator $x_k$ and let $K$ be the subgroup 
\[
K = \langle x_k^\ell, x_k^i x_j x_k^{-i} \: (0 \leq j \leq k-1, \: i \in \mathbb{Z}) \rangle.
\]
Then $K$ is normal, and $G / K \cong \pres{Gp}{x_k}{x_k^\ell = 1}$. Hence $K$ has index $\ell$ in $G$, and a Schreier transversal for $K$ in $G$ is given by $T = \{ 1, x_k, \dots, x_k^{\ell-1} \}$. We will prove that $H \cong K$, yielding the result.

Letting $s_{i,j} = x_k^i x_j \overline{x_k^i x_j}^{-1}$, then $K$ is generated by the $s_{i,j}$ for $0 \leq i \leq \ell-1$ and $0 \leq j \leq k$. If $j < k$, then $x_k^i x_j x_k^{-i} \in H$, so $\overline{x_k^i x_j} = x_k^i$ for all $0 \leq i \leq \ell$. If $j=k$, then $\overline{x_k^i x_j} = x_k^{i+1 \pmod \ell}$. Hence, we find the generators $s_{i,j}$ to be 
\[
s_{i,j} = \begin{cases*} 1 & if $j=k$ and $0 \leq i < \ell-1$, \\
x_k^\ell, & if $j=k$ and $i=\ell-1$, \\ x_k^i x_j x_k^{-i} & if $0 \leq j \leq k-1$.
\end{cases*}
\]
Let $\tau$ be the Reidemeister-Schreier rewriting process which rewrites an element of $K$ into the generators $s_{i,j}$. Let $\alpha = x_k^\ell$. For $0 \leq i \leq \ell-1$, let $\beta_i = x_k^i x_0 x_k^{-i}$, and for $1 \leq j \leq k-1$ let $\beta_{i,j} = x_k^i x_j x_k^{-i}$. Then $K$ is generated by $\alpha$ and the $x_i, x_{i,j}$, and has as defining relators $\tau(x_k^i [x_0, x_j^\ell] x_k^{-i})$ for $0 \leq i \leq \ell-1$ and $1 \leq j \leq k$. 

Using the Reidemeister-Schreier method to find $\tau(x_k^i [x_0, x_j^\ell] x_k^{-i})$, we first consider the case when $j=k$, and find, by simple insertions of $x_k^{-i} x_k^i$-pairs, that 
\[
\tau(x_k^i [x_0, x_k^\ell] x_k^{-i})  = \tau\left((x_k^i x_0 x_k^{-i})(x_k^\ell)(x_k^ix_0^{-1} x_k^{-i}) (x_k^{-\ell}) \right) = \beta_i \alpha \beta_i^{-1} \alpha^{-1} = [\beta_i, \alpha].
\]
If, on the other hand, $j\neq k$, then 
\begin{align*}
\tau\left(x_k^i [x_0, x_j^\ell] x_k^{-i}  \right) &\equiv \tau\left( x_k^i x_0 x_j^\ell x_0^{-1} x_j^{-\ell} x_k^{-i} \right) \\
&= \tau\left( x_k^i x_0 (x_k^{-i} x_k^i) x_j^\ell (x_k^{-i} x_k^i) x_0^{-1} (x_k^{-i} x_k^i) x_j^{-\ell} x_k^{-i} \right) \\
&= \tau\left( (x_k^i x_0 x_k^{-i})(x_k^i x_j^\ell x_k^{-i})(x_k^i x_0^{-1} x_k^{-i})(x_k^i x_j^{-\ell} x_k^{-i}) \right) \\
&= \beta_i \beta_{i,j}^\ell \beta_i^{-1} \beta_{i,j}^{-\ell} = [\beta_i, \beta_{i,j}^\ell].
\end{align*}
Thus a presentation for $K$ is 
\[
K \cong \pres{Gp}{\alpha, \{ \beta_i, \beta_{i,j}\}_{i,j}}{[\alpha, \beta_i] = [\beta_i, \beta_{i,j}] = 1, \quad \text{for all $0 \leq i \leq \ell-1$ and $1 \leq j \leq k-1$.}}
\]
This group is now isomorphic to $H = B(S^{(b)}_{k,\ell})$, and we are done.
\end{proof}

Note that if $k=1$, then $B(S_{k,\ell}) = \BS(\ell,\ell)$, and the submonoid membership problem is decidable, as $\BS(\ell,\ell)$ has decidable rational subset membership problem, being a finite extension of the right-angled Artin group $F_n \times \mathbb{Z}$ which has decidable rational subset membership problem by Theorem~\ref{Thm:Lohrey-Steinberg}, and decidability of this problem is preserved by taking finite extensions \cite{Grunschlag1999}. If instead $\ell=1$, then $B(S_{k,\ell})$ is a right-angled Artin group defined by a star graph, and has decidable rational subset membership problem by Theorem~\ref{Thm:Lohrey-Steinberg}. Hence, we find the following characterisation. 

\begin{corollary}\label{Cor:BS-star-boundary-for-SMMP}
Let $\mathcal{P}$ be either the submonoid membership problem, or the rational subset membership problem. Then $\mathcal{P}$ is decidable in $B(S_{k,\ell})$ if and only if $k=\ell=1$.
\end{corollary}

In view of Theorem~\ref{Thm:Lohrey-Steinberg}(3), the following question becomes natural:

\begin{question}
Is the submonoid membership problem recursively equivalent to the rational subset membership problem in all right-angled Baumslag-Solitar-Artin groups?
\end{question}

We now turn to a brief study of the subgroup membership problem.

\subsection{The subgroup membership problem}\label{Subsec:SGMP}

Kambites \cite{Kambites2009} has proved that if $\Gamma$ is a chordal graph (i.e. $C_4$-free), then $A(\Gamma)$ does not contain a subgroup isomorphic to $A(C_4)$. We now exhibit a chordal graph $\Gamma_2$ such that $B(\Gamma_2)$ contains a finite index subgroup isomorphic to $A(\Gamma_2')$ with $\Gamma_2'$ not chordal. This demonstrates even further that the subgroup structure of \rabsags{} seems significantly more complicated than that of right-angled Artin groups.

\begin{proposition}\label{Prop:Chordal-RABSAG-contains-C4}
Let $G = B(\Gamma_2)$, and let $H = A(\Gamma_2')$, where $\Gamma_2, \Gamma_2'$ are the graphs
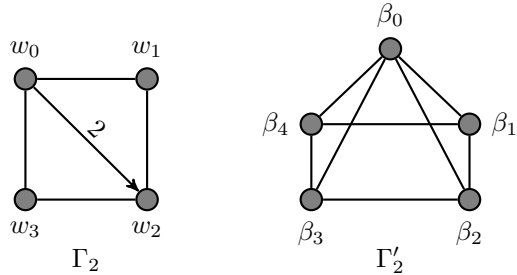
\begin{center}
\begin{tikzpicture}[>=stealth',thick,scale=0.8,el/.style = {inner sep=2pt, align=left, sloped}]%

                        \node (w0)[label=above:$w_0$][circle, draw, fill=black!50,
                        inner sep=0pt, minimum width=8pt] at (-1,1) {};
                        \node (w1)[label=above:$w_1$][circle, draw, fill=black!50,
                        inner sep=0pt, minimum width=8pt] at (1,1) {};
                        \node (w2)[label=below:$w_2$][circle, draw, fill=black!50,
                        inner sep=0pt, minimum width=8pt] at (1,-1) {};
                        \node (w3)[label=below:$w_3$][circle, draw, fill=black!50,
                        inner sep=0pt, minimum width=8pt] at (-1,-1) {};
\path[->] 
    (w0)  edge node[el,above]{$2$}         (w2);
\path[-]
	(w0)  edge node[above]{}         (w1)
	(w1)  edge node[above]{}         (w2)
	(w2)  edge node[above]{}         (w3)
	(w3)  edge node[above]{}         (w0);
	\node (lab) at (0, -2) {$\Gamma_2$};
	
\node (b0)[label=above:$\beta_0$][circle, draw, fill=black!50,                        						inner sep=0pt, minimum width=8pt] at (5,1.5) {};
\node (b1)[label=right:$\beta_1$][circle, draw, fill=black!50,
                        inner sep=0pt, minimum width=8pt] at (6.3,0.25) {};
\node (b2)[label=below:$\beta_2$][circle, draw, fill=black!50,
                        inner sep=0pt, minimum width=8pt] at (6.3,-1) {};
\node (b3)[label=below:$\beta_3$][circle, draw, fill=black!50,
                        inner sep=0pt, minimum width=8pt] at (3.7,-1) {};
\node (b4)[label=left:$\beta_4$][circle, draw, fill=black!50,
                        inner sep=0pt, minimum width=8pt] at (3.7,0.25) {};

\path[-]
	(b0)  edge node[above]{}         (b1)
	(b0)  edge node[above]{}         (b2)
	(b0)  edge node[above]{}         (b4)
	(b0)  edge node[above]{}         (b3)
	(b1)  edge node[above]{}         (b2)
	(b1)  edge node[above]{}         (b4)
	(b2)  edge node[above]{}         (b3)
	(b4)  edge node[above]{}         (b3);
	\node (lab) at (5, -2) {$\Gamma_2'$};
\end{tikzpicture}
\end{center}
respectively. Then $H$ is isomorphic to an index $2$ subgroup of $G$, and $H \cong \mathbb{Z} \times F_2 \times F_2$. In particular, $A(C_4) \leq G$, and the subgroup membership problem is undecidable in $G$. 
\end{proposition}
\begin{proof}
The proof follows along the same pattern as that of Proposition~\ref{Prop:BGamma_1-contains-P4}. Taking 
\[
H = \langle w_1, w_2^2, w_3, w_2^i w_0 w_2^{-i} \: (i \in \mathbb{Z}) \rangle,
\]
we find that $H$ is an index $2$ normal subgroup of $G$, with Schreier transversal $T = \{ 1, w_2 \}$. Just as in the proof of Proposition~\ref{Prop:BGamma_1-contains-P4}, we enumerate the generators and relators of $H$, and find that all such relators are commutators of generators, and is hence a right-angled Artin group, and that the underlying graph of this group is $\Gamma_2'$. We leave the details to the reader. 
\end{proof}
\begin{remark}
The subgroup membership problem is undecidable in general for two-relator groups, but is open for one-relator groups. It would be interesting to find a boundary for \rabsags{} for when the subgroup membership problem is decidable, in the spirit of the boundary given by Theorem~\ref{Thm:Lohrey-Steinberg} for right-angled Artin groups. 
\end{remark}

Recall that $B(P_{3,k})$ is the group defined by 
\[
B(P_{3,k}) = \pres{Gp}{a,b,c}{[a,b^k] = 1, [a,c] = 1}.
\]
This group has undecidable submonoid membership problem (Proposition~\ref{Prop:BGamma_1-contains-P4}) if and only if $k \geq 2$. Recall also the definition of the $\ell$-regular Baumslag-Solitar-Artin star with $k$ arms $S_{k,\ell}$, and that this has undecidable submonoid membership problem when $k \geq 2$ and $\ell \geq 2$.

\begin{question}\label{Question:SGMP-in-certain-rabsags}
\begin{enumerate}
\item Is the subgroup membership problem decidable in $B(P_{3,k})$ for $k \geq 2$? 
\item Is the subgroup membership problem decidable in $B(S_{k,\ell})$ for $k, \ell \geq 2$?
\end{enumerate}
\end{question}

We note that, as decidability of the subgroup membership problem is preserved under taking finite extensions \cite{Grunschlag1999}, a positive answer to Question~\ref{Question:SGMP-in-certain-rabsags}(2) would (by the results of \S\ref{Subsec:Moldavanskii-groups}) solve the subgroup membership problem in the one-relator groups
\[
G_{m,n} = \pres{Gp}{a,b}{[a^m, b^n] = 1},
\]
for which the author is at present not aware of any solution to the subgroup membership problem. 

Finally, we remark on a connection with the Diophantine problem. Ciobanu, Holt \& Rees proved (see \cite[Theorem~3.1]{Ciobanu2020}) that the Diophantine problem is decidable in any group which is virtually a direct product of hyperbolic groups. In particular, the Diophantine problem is decidable in any group which is virtually a right-angled Artin group $A(\Gamma)$ when $\Gamma$ is a star graph, as in this case $A(\Gamma) \cong F \times \mathbb{Z}$ for a finitely generated free group $F$. In particular (as we shall see in \S\ref{Subsec:Moldavanskii-groups}), this applies to, and solves the Diophantine problem in, the unimodular Baumslag-Solitar groups $\BS(m,m)$, which are virtually $F_m \times \mathbb{Z}$ when $m >0$. Furthermore, this solves the Diophantine problem in the \rabsag{} $B(\Gamma_2)$ in Proposition~\ref{Prop:Chordal-RABSAG-contains-C4}. We give the following conjecture in light of this:

\begin{conjecture}\label{Conj:all_rabsag_virtually_star_dec_DP}
The Diophantine problem is decidable in any group $G$ which is virtually a \rabsag{} $B(\Gamma)$ where the underlying undirected graph of $\Gamma$ is a star graph. 
\end{conjecture}

If this conjecture turns out to be true, then this would solve the Diophantine problem in a large class of one-relator groups, see \S\ref{Subsec:Moldavanskii-groups}. As a particular case of Conjecture~\ref{Conj:all_rabsag_virtually_star_dec_DP}, it would be interesting to know whether the Diophantine problem is decidable in the two-relator group $B(\Gamma_1)$ from Proposition~\ref{Prop:BGamma_1-contains-P4}, i.e. 
\[
B(\Gamma_1) \cong \pres{Gp}{x,y,z}{[x,y^2] = [x,z] = 1}.
\]
It bears repeating that the submonoid membership problem is undecidable in $B(\Gamma_1)$.

\subsection{Moldavanskii groups}\label{Subsec:Moldavanskii-groups}

\noindent In the following section, we will investigate a rather peculiar class of one-relator groups $G_{m,n}$, and relate the Diophantine problem for these groups to the Diophantine problem in certain \rabsags. We shall call these groups \textit{Moldavanskii-Tieudjo groups}, after the two authors who systematically investigated some of their properties \cite{Tieudjo1998, Tieudjo2005, Tieudjo2006, Tieudjo2008, Tieudjo2010, Tieudjo2010b}. 

Let $m, n \geq 1$. Then we define
\begin{equation}
G_{m,n} = \pres{Gp}{a,b}{[a^m, b^n] = 1}.
\end{equation}
and call $G_{m,n}$ a \textit{Moldavanskii-Tieudjo} group. Some facts are worth mentioning. Obviously, if $m=1$, then $G_{m,n} \cong \BS(n,n)$. Thus these groups arise as natural generalisations of the Baumslag-Solitar groups. Let $m,n>1$. Then $G_{m,n}$ is residually finite (this is an easy consequence of a theorem due to Baumslag \cite{Baumslag1983}); they are conjugacy separable \cite{Kim1995}, see also \cite{Tieudjo2010b}, and so in particular have decidable conjugacy problem; and the finitely generated abelian subgroups of $G_{m,n}$ are finitely separable \cite{Tieudjo2005}. Furthermore, using the Magnus-Moldavanskii hierarchy the group $G_{m,n}$ is an HNN-extension of the group $T_{m,m} = \pres{Gp}{a,b}{a^m b^m = 1}$, as studied in \S\ref{Sec:Warmup-torus-knots}. 

Using, again, the Reidemeister-Schreier method, we shall describe some finite index subgroups inside $G_{m,n}$. In particular, we will find that $G_{m,n}$ is virtually a \rabsag{} for all $m, n > 1$. Recall the definition of the Baumslag-Solitar Artin stars $S_{k,\ell}$ as defined in Figure~\ref{Fig:BS-star}. 

\begin{theorem}\label{Thm:[am,bn]-contains-BS-star}
Let $G_{m,n} = \pres{Gp}{a,b}{[a^m, b^n] = 1}$ with $m, n \geq 1$. Then $G_{m,n}$ contains an index $m$ normal subgroup isomorphic to $B(S_{m,n})$.
\end{theorem}
\begin{proof}
Consider the subgroup $H$ of $G$ generated by 
\[
H = \langle a^m, a^i b a^{-i} \: (i \in \mathbb{Z}) \rangle.
\]
Then $H$ is, despite first appearances, finitely generated, as any conjugate $a^i b a^{-i}$ is a product of powers of $a^m$ and $a^j b a^{-j}$ with $0 \leq j < m$. Furthermore, $H$ is evidently normal, and has finite index in $G$. Note that neither of the previous properties have any dependency on $G$ itself, but are true in any group generated by two elements $a$ and $b$. Now, 
\[
G/H \cong \pres{Gp}{a,b}{[a^m, b^n] = 1, b = 1, a^m = 1} \cong \pres{Gp}{a}{a^m =1}.
\]
Thus $H$ has index $m$ in $G$. A Schreier transversal for $H$ in $G$ is given by $T = \{ 1, a, \dots, a^{m-1} \}$. Using the standard Reidemeister-Schreier technique, we find a generating set 
\[
S = \{ tx \overline{tx}^{-1} \mid t \in T, x \in \{a, b\}\}
\]
for $H$. Let $s_{i,x}$ denote $a^i x \overline{a^ix}^{-1}$ for $0 \leq i \leq m-1$ and $x \in \{ a, b\}$. Then 

\[
s_{i,a} = \begin{cases*} 1 & if $0 \leq i < m-1$ \\ a^m & if $i = m-1$ \end{cases*}
\]
and, as $a^iba^{-i} \in H$ and thus $a^ib \in Ha^i$ for all $i \in \mathbb{Z}$, we have $\overline{a^ib} = a^{i}$. Hence 
\[
s_{i,b} = a^i b \overline{a^ib}^{-1} = a^i b a^{-i}.
\]
Let $\alpha$ denote the generator $s_{m-1,a}$, and let $\beta_{i+1}$ denote the generator $s_{i,b}$ for all $0 \leq i \leq m-1$. Then $H$ is generated by the set $\alpha, \beta_1, \beta_2, \dots, \beta_m$, and we can find a presentation for $H$ as usual. Let $\tau$ be a rewriting process for $\{a,b\}$-words to words over the new alphabet. Then a set of defining relators for $H$ is $\tau(a^i [a^m, b^n] a^{-i})$ for $0 \leq i \leq m-1$. But
\begin{align*}
\tau(a^i [a^m, b^n]a^{-i}) &= \tau\left(a^m a^i b^n a^{-m} b^{-n} a^{-i} \right) \\&= \tau\left( (a^m) (a^i b a^{-i})^n (a^m)^{-1} (a^i b a^{-i})^{-n} \right) \\ &= \alpha \beta_i^n \alpha^{-1} \beta_i^{-n}.
\end{align*}
That is, $H$ is defined by the presentation
\[
H \cong \pres{Gp}{\alpha, \beta_1, \dots, \beta_m}{[\alpha, \beta_i^n] = 1 \: (1 \leq i \leq m)},
\]
i.e. $H \cong B(S_{m,n})$. This is what was to be shown. 
\end{proof}

Taking the special case of $n=1$ above, as $G_{m,1} = \BS(m,m)$ and as $B(S_{m,1}) = F_m \times \mathbb{Z}$, the above Theorem~\ref{Thm:[am,bn]-contains-BS-star} provides a direct proof that $\BS(m,m)$ is virtually $F_m \times \mathbb{Z}$ for $m \geq 1$. This is not a new theorem, by any means -- see e.g. \cite{Levitt2015}, cf. also \cite{Meskin1972, Wise2000}. A direct proof via Bass-Serre theory is also not particularly difficult. For our purposes, we will however write out the following direct corollary:

\begin{proposition}\label{Prop:BS(m,m)-dec}
The Diophantine problem is decidable in $G_{m,1} \cong \BS(m,m)$ for all $m \in \mathbb{Z}$. 
\end{proposition}

This has already been (implicitly) noted by Ciobanu, Holt \& Rees \cite[p.91]{Ciobanu2020}. A similar reasoning also solves the Diophantine problem in $\BS(m,-m)$. The reason for writing this out explicitly is to place Proposition~\ref{Prop:BS(m,m)-dec} in the wider context of the groups $G_{m,n}$. Indeed, if Conjecture~\ref{Conj:all_rabsag_virtually_star_dec_DP} holds, then Theorem~\ref{Thm:[am,bn]-contains-BS-star} implies that all groups $\pres{Gp}{a,b}{[a^m, b^n]=1}$ have decidable Diophantine problem. 

We now use Theorem~\ref{Thm:[am,bn]-contains-BS-star} to prove some corollaries. The first is an application to the submonoid membership problem in one-relator groups. As proved in Proposition~\ref{Prop:Star-BS-contains-P3ell} and Corollary~\ref{Cor:BS-star-boundary-for-SMMP}, the groups $B(S_{k,\ell})$ all have undecidable submonoid membership problem for $k, \ell \geq 2$. Hence, we conclude:

\begin{corollary}\label{Cor:Gmn-has-undec-SMMP}
Let $G_{m,n} = \pres{Gp}{a,b}{[a^m, b^n] = 1}$ with $m, n > 1$. Then the submonoid membership problem is undecidable in $G_{m,n}$. 
\end{corollary}

The groups $G_{m,n}$ in Corollary~\ref{Cor:Gmn-has-undec-SMMP} are not the first to be proved to have undecidable submonoid membership problem. Indeed, Gray \cite{Gray2020} shows that the submonoid membership problem is undecidable in 
\begin{equation}\label{Eq:Bob-group}
B = \pres{Gp}{a,t}{[a, tat^{-1}] = 1} \cong \pres{Gp}{a,b}{[ab, ba]=1}.
\end{equation}
Now, defining $B_i = \pres{Gp}{a,t}{[a, t^i a t^{-i}] = 1}$ for $i \geq 1$, we have $B = B_1$, and $B_{i+1}$ is an HNN-extension of $B_i$ conjugating the subgroup $\langle a \rangle$. Just as in \eqref{Eq:Bob-group}, we can furthermore define $B_i$ as a one-relator group with defining relator a commutator of two positive words, i.e. 
\begin{equation}
B_i \cong \pres{Gp}{a,b}{[ab^i, b^ia] = 1} \quad \text{for $i \geq 1$.}\label{Eq:Bob-family}
\end{equation}
As $B_i \leq B_{i+1}$ for all $i \geq 1$, the submonoid membership problem is now undecidable in every (one-relator) group $B_i$. However, the defining relator $[ab^i, b^ia]$ is not a commutator $[u, v]$ of two words $u, v$ over disjoint alphabets. In particular, $B_i$ is not obviously decomposable a free product of two free groups with commuting subgroups. 

Recall (see \cite[p. 220, ex. 22]{Magnus1966}) that if $A, B$ are two groups with $H \leq A$ and $K \leq B$, then the group $\pres{Gp}{A, B}{[H, K] = 1}$ is called the \textit{free product of $A$ and $B$ with commuting subgroups $H$ and $K$}. Such products have seen some recent study, see e.g. \cite{Sokolov2014}. 

Note that the groups $G_{m,n}$ have the structure of a the free product of two infinite cyclic groups $\langle a \rangle$ and $\langle b \rangle$ with commuting subgroups $H = \langle a^m \rangle$ and $K = \langle b^n \rangle$ (both of finite index). Hence, Corollary~\ref{Cor:Gmn-has-undec-SMMP} gives the following additional corollary:

\begin{corollary}\label{Cor:FP-with-commuting-can-be-undec}
There exists a free product $G$ of two infinite cyclic groups with commuting subgroups such that the submonoid membership problem in $G$ is undecidable. 
\end{corollary}

To the best of the author's knowledge, this the first known occurrence of such groups.

Finally, let $G_1, G_2$ be two groups with decidable submonoid membership problem. As mentioned in the remark following Theorem~\ref{Thm:Lohrey-Steinberg}, it is currently unknown whether $G_1 \ast G_2$ always has decidable submonoid membership problem. Corollary~\ref{Cor:FP-with-commuting-can-be-undec} does not give an answer to this question, but shows that already rather complicated nature of going beyond free products. It would be interesting to know whether the subgroup membership problem can be similarly undecidable in this situation.

\subsection{Application: subgroups of Baumslag-Solitar-Artin groups}\label{Subsec:applying-MT-to-rabsags}

The second application of Theorem~\ref{Thm:[am,bn]-contains-BS-star} will be a small insight into the subgroup structure of the \rabsag{} $B(P_{3,k})$ from Proposition~\ref{Prop:BGamma_1-contains-P4}. The way this is proved is via knowing the \raag{} subgroups of one-relator groups. Using our results on the Moldavanskii-Tieudjo groups $G_{m,n}$ from \S\ref{Subsec:Moldavanskii-groups}, we can deduce the following result about the \rabsag{} $B(P_{3,k})$, which makes no mention of one-relator groups.

\begin{corollary}\label{Cor:P3k-contains-only-good-raags}
Let $G = B(P_{3,k})$, where $k \geq 2$ and $P_{3,k}$ is the graph 
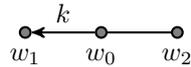
\begin{center}
\begin{tikzpicture}[>=stealth',thick,scale=1.0,el/.style = {inner sep=2pt, align=left, sloped}]%

                        \node (w0)[label=below:$w_0$][circle, draw, fill=black!50,
                        inner sep=0pt, minimum width=4pt] at (0,0) {};
                        \node (w1)[label=below:$w_1$][circle, draw, fill=black!50,
                        inner sep=0pt, minimum width=4pt] at (-1,0) {};
                        \node (w2)[label=below:$w_2$][circle, draw, fill=black!50,
                        inner sep=0pt, minimum width=4pt] at (1,0) {};
                        
\path[->] 
    (w0)  edge node[above]{$k$}         (w1);
\path[-]
	(w2)  edge node[above]{}         (w0);
\end{tikzpicture}
\end{center}
and let $\Gamma$ be any graph. Then $G$ contains the right-angled Artin group $A(\Gamma)$ as a subgroup if and only if $\Gamma$ is a forest.  
\end{corollary}
\begin{proof}
$(\impliedby)$. Let $\Gamma$ be any finite forest. Then, by \cite[Theorem~1.8]{Kim2013}, $A(\Gamma)$ embeds in $A(P_4)$. But $G$ contains a copy of $A(P_4)$ by Proposition~\ref{Prop:BGamma_1-contains-P4} and the remark following it.

$(\implies)$. Let $\Gamma$ be a graph such that $A(\Gamma)$ embeds in $G$. Now, $G$ embeds in $A(S_{k,k}^{(b)})$, where $S_{k,k}^{(b)}$ is as in Figure~\ref{Fig:BS-star}, as $k \geq 2$ and $P_{3,k}$ is an induced subgraph of $S^{(b)}_{k,k}$. By Proposition~\ref{Prop:Star-BS-contains-P3ell}, $B(S^{(b)}_{k,k}) \leq B(S_{k,k})$, where $S_{k,k}$ is also shown in Figure~\ref{Fig:BS-star}. Now, by Theorem~\ref{Thm:[am,bn]-contains-BS-star}, the group $G_{k,k} = \pres{Gp}{a,b}{[a^k, b^k]=1}$ contains a subgroup isomorphic to $B(S_{k,k})$. Hence we have shown that 
\[
A(\Gamma) \leq G \leq B(S_{k,k}^{(b)}) \leq B(S_{k,k}) \leq G_{k,k}.
\]
As $G_{k,k}$ is a one-relator group, $\Gamma$ is necessarily a finite forest by Howie's lemma (Lemma~\ref{Lem:Howie'sLemma}). 
\end{proof}

In particular, by Corollary~\ref{Cor:P3k-contains-only-good-raags}, the group $B(P_{3,k})$ does not contain any subgroup isomorphic to $F_2 \times F_2$. This lends some credence to the following conjecture:

\begin{conjecture}
The subgroup membership problem is decidable in $B(P_{3,k})$ for every $k \geq 2$. 
\end{conjecture}

It seems an interesting challenge to classify precisely those right-angled Baumslag-Solitar-Artin groups $B(\Gamma)$ which contain $F_2 \times F_2$ as a subgroup. As proved in Proposition~\ref{Prop:Chordal-RABSAG-contains-C4}, there exist chordal graphs $\Gamma$ for which $F_2 \times F_2$ can nevertheless be embedded in $B(\Gamma)$, unlike the case for right-angled Artin groups (by Kambites' theorem).

\section{Torsion subgroups}\label{Sec:Torsion-subgroups}

\noindent Recall the fact from \S\ref{Subsec:OR-groups} that every finitely generated subgroup of a hyperbolic one-relator group is hyperbolic, and hence has decidable Diophantine problem. In this section, we will exhibit some classes torsion-free one-relator groups which embed into one-relator groups with torsion, thus solving the Diophantine problem in certain torsion-free one-relator groups. Note that it follows from Pride \cite{Pride1977}, every torsion-free two-generator subgroup of a one-relator group with torsion is free. Hence, this method can only give new results in the case of three or more generators. Using this method, in \S\ref{Subsec:Squared-torus-knots}, we give an example of an infinite class of torsion-free one-relator groups with decidable Diophantine problem. Following this, we generalise an example due to B.~B.~Newman (Example~\ref{Ex:Newman}) to find another infinite family of torsion-free one-relator groups with decidable Diophantine problem. The overall aim of this section is to demonstrate the usefulness of this embedding method; the classes obtained here are by no means the only new classes obtainable in this manner.

\subsection{Squared torus knots}\label{Subsec:Squared-torus-knots}

We begin, similarly to \S\ref{Sec:Warmup-torus-knots}, by an example arising from torus knot groups. Let $p, q,k \geq 2$, and let 
\begin{equation}\label{Eq:inflated-torus-knot}
T^{(k)}_{p,q} = \pres{Gp}{a,b}{(a^p b^q)^k = 1}.
\end{equation}
These groups are ``inflated'' versions of the groups $T_{p,q}$ studied in \S\ref{Sec:Warmup-torus-knots}. As $k > 1$, all one-relator groups \eqref{Eq:inflated-torus-knot} have torsion and are hyperbolic. In particular, the groups $T^{(k)}_{p,q}$ all have decidable Diophantine problem for $m, n, k \geq 1$ (the case $k=1$ is dealt with in \S\ref{Sec:Warmup-torus-knots}). 

Despite the simplicity of the presentations of these groups, we can find some interesting, seemingly unrelated, torsion-free one-relator groups inside, even in some restricted cases. Let $m \geq 0$, and let
\begin{equation}\label{Eq:TF-from-squared-torus}
H_m = \pres{Gp}{a,b,c}{(ac^{m+1}a)(bc^mb) = 1}.
\end{equation}
Then $H_m$ is a torsion-free one-relator group.

\begin{proposition}\label{Prop:TF-in-squared-torus}
Let $m \geq 1$. Then $T_{2,2m+1}^{(2)}$ has an index $2$ subgroup isomorphic to $H_m$. 
\end{proposition}
\begin{proof}
Let $H = \langle a^2, ab, ab^{-1} \rangle$. Clearly, $H$ is normal, as can be seen e.g. by constructing the Stallings graph for the subgroup of the free group $F(a,b)$ generated by the same generating set. This graph has two vertices; in particular $H$ has index $2$ in $G$. A Schreier transversal for $H$ in $G$ is given by $\{ 1, a \}$. As usual, we find the Schreier generators, which in this case become $ba^{-1}, ab, a^2$. Call these $\alpha, \beta$, resp. $\gamma$. Then, using a Reidemeister--Schreier rewriting $\tau$ to rewrite the relator of $T_{2,2m+1}^{(2)}$ and its conjugate by $a$, we find:
\begin{align*}
\tau\left(a^2b^{2m+1}a^2b^{2m+1}\right) &\equiv \tau\left((a^2) (b^2)^{m} b a a b (b^2)^{m}\right) \\ &\equiv \tau\left((a^2) (ba^{-1}ab)^{m} (ba^{-1})(a^2) (ab) (ba^{-1} ab)^{m}\right) \\
&= \gamma (\alpha\beta)^m \alpha \gamma \beta (\alpha \beta)^m =: R_1 \\
\tau\left(aa^2b^{2m+1}a^2b^{2m+1}a^{-1}\right) &\equiv \tau\left((a^2)(ab) (b^2)^m (a^2) (b^2)^m ba^{-1} \right) \\
&= \gamma \beta (\alpha\beta)^m \gamma (\alpha\beta)^m \alpha =: R_2.
\end{align*}
As $R_1$ is a cyclic conjugate of $R_2$, $H$ is a one-relator group:
\[
H \cong \pres{Gp}{\alpha, \beta, \gamma}{\gamma (\alpha\beta)^m \alpha \gamma \beta (\alpha \beta)^m = 1}.
\]
The automorphism of the free group on $\alpha, \beta, \gamma$ defined by $\alpha \mapsto \alpha \beta^{-1}, \beta \mapsto \beta$, and $\gamma \mapsto \gamma \beta$ gives a new presentation for $H$ as 
\[
H \cong \pres{Gp}{\alpha, \beta, \gamma}{\beta \gamma \alpha^{m+1} \gamma \beta \alpha^m = 1}.
\]
By a cyclic shift of the relator, and relabelling the generators, this is \eqref{Eq:TF-from-squared-torus}, so $H \cong H_m$.
\end{proof}

\begin{corollary}\label{Cor:acm+1a=bcmb_dio}
The Diophantine problem is decidable in the torsion-free one-relator group
\[
H_m = \pres{Gp}{a,b,c}{(ac^{m+1}a)(bc^mb) = 1}
\]
for every $m \geq 0$.
\end{corollary}

We remark that as the groups $H_m$ are positive one-relator groups, they are residually solvable \cite{Baumslag1971}. It would be interesting to know whether the Diophantine problem is decidable in all positive one-relator groups; this can be related to the Diophantine problem in certain one-relation monoids \cite{Garreta2021b}. As $H_m$ is virtually a one-relator group with torsion, it is also coherent \cite{Louder2020}.

\begin{remark}
The groups $\pres{Gp}{a,b}{(a^m b^n)^t = 1}$ were studied by Baumslag \cite[Lemma~1]{Baumslag1967} in a Bulletin article, who proved (or, rather, claimed) that these are all virtually residually free. The author is not aware of any published article where the proof of this fact appears. 
\end{remark}

\subsection{Newman One-relator Groups}\label{Subsec-newman-groups}

As mentioned in \S\ref{Subsec:OR-groups}, every $2$-generator subgroup of a one-relator group with torsion is either free or has torsion. B. B. Newman and S. Pride ``once thought'' (\cite[p. 483]{Pride1977}) that \textit{every} finitely generated subgroup of a one-relator group with torsion will either be free or have torsion. However, Pride reports on a counterexample to this statement, discovered by Newman. In fact, the example appears already in Newman's Ph.D. thesis (see \cite{NybergBrodda2021a} for the story of how this thesis was uncovered):

\begin{example}[{B. B. Newman, \cite[p. 72]{Newman1968}}]\label{Ex:Newman}
Let $G = \pres{Gp}{x,y}{(yxy^{-1}x^{-2})^2 = 1}$. Then the subgroup $H = \langle y, x^2, xyx^{-1} \rangle$ has index $2$ in $G$, and 
\begin{equation}\label{Eq:Newman'sGroup}
H \cong \pres{Gp}{a,b,c}{a^b a^c = a},
\end{equation}
where we use the convention that $a^b = bab^{-1}$. In particular, as $H$ is hyperbolic, being a finite index subgroup of a hyperbolic group (or applying Corollary~\ref{Cor:f.g.-ORwt-is-hyperbolic}), the Diophantine problem is decidable in the torsion-free one-relator group $\pres{Gp}{a,b,c}{a^ba^c = a}$. We remark that $a$ is obviously contained in every term of the lower central series of $H$, and hence $H$ cannot be residually nilpotent; hence, in particular, $H$ is not a free group \cite{Magnus1935}.
\end{example}

We will now, in various ways, generalise this example to find more torsion-free one-relator groups with decidable Diophantine problem. We define two new families of one-relator groups. Every group in the first family has torsion; in the second, every group is torsion-free. Let $k \geq 1$, and $m, n \in \mathbb{Z}$. Then the group
\begin{equation}\label{Eq:BSkmn-def-inflated}
\BS^{(k)}(m,n) = \pres{Gp}{a,b}{(ba^m b^{-1}a^{-n})^k = 1}
\end{equation}
will be called a \textit{$k$-fold inflated Baumslag-Solitar group}. Of course, $\BS^{(1)}(m,n) = \BS(m,n)$. If $k>1$, then $\BS^{(k)}(m,n)$ is a one-relator group with torsion, and in particular is hyperbolic. 

The second family of groups are defined analogous to B. B. Newman's group \eqref{Eq:Newman'sGroup}. Let $p,q \geq 1$ be integers. Then the group 
\begin{equation}\label{Eq:NewmanPridegroup-def}
\NP(p,q) = \pres{Gp}{y, x_1, x_2, \dots, x_p}{y = (y^q)^{x_1} (y^q)^{x_2} \cdots (y^q)^{x_p}}
\end{equation}
will be called a \textit{Newman group}. The class of Newman groups generalise the solvable Baumslag-Solitar groups $\BS(1,n)$, as 
\[
\NP(1, q) = \pres{Gp}{y, x_1}{y = (y^q)^{x_1}} \cong \BS(1, q). 
\]
Of course, $\NP(p,q)$ is not solvable by $p>1$, as by the \textit{Freiheitssatz} such groups contain free subgroups of rank $2$ (indeed, $\NP(p,q)$ virtually surjects $F_2$ when $p>1$, by \cite{Baumslag1978}).

\begin{proposition}\label{Prop:Newman-group-in-inflated}
For every $p \geq 2$ and $q \geq 1$, the Newman group $\NP(p, q)$ is isomorphic to an index $p$ normal subgroup of the inflated Baumslag-Solitar group $\BS^{(p)}(1,pq)$.
\end{proposition}
\begin{proof}
We will let $G = \BS^{(p)}(1,pq)$, with presentation as in \eqref{Eq:BSkmn-def-inflated}, and let $A =\{ a, b\}$. If $p=1$,
\[
\BS^{(p)}(1,pq) = \BS(1, q) \cong \pres{Gp}{x_1, x_2}{x_1 = (x_1^q)^{x_2}} = \NP(1, q) = \NP(1,pq),
\]
as required. Suppose therefore $p>1$. Let $H = \langle a^p, a^i b a^{-i} \: (i \in \mathbb{Z} \rangle$. As is easily verified, $H$ is normal in $G$. Furthermore, 
\begin{align*}
G / H &\cong \pres{Gp}{a,b}{(bab^{-1}a^{-pq})^p = 1, a^p = 1, a^i ba^{-i} = 1 \: (0 \leq i \leq p-1)} \\
&= \pres{Gp}{a,b}{b a^pb^{-1} = 1, a^p = 1, b=1} = \pres{Gp}{a}{a^p = 1} \cong C_p.
\end{align*}
In particular $H$ has index $p$ in $G$. Thus, it suffices to prove $H \cong \NP(p, q)$. 

We find a presentation for $H$ by the Reidemeister--Schreier method. By the presentation for $G/H$ above, a Schreier transversal for $H$ in $G$ is $T = \{ 1, a, \dots, a^{p-1} \}$. As usual, the Schreier generators $tx \overline{tx}^{-1}$ for $t \in T$ and $x \in A$ can be found to be $a^p$ and $b^i a b^{-i}$ for $0 \leq i \leq p-1$. Denote these by $\alpha$ resp. $\beta_i$. Let $R = a^i(bab^{-1}a^{-pq})a^{-i}$. We will now rewrite the conjugates of $R$ over these generators. This is straightforward; indeed, note that if $p=i+1$ then 
\begin{equation}\label{Eq:single-step-of-newman}
a^i R a^{-i} \equiv a^i bab^{-1}a^{-pq} a^{-i} = \underbrace{(a^iba^{-i})}_{\beta_i}\underbrace{(a^{i+1}b^{-1}a^{-(i+1)}}_{\beta_{i+1}^{-1}})\underbrace{(a^p)^{-q}}_{\alpha^{-q}} a
\end{equation}
and hence, starting for simplicity with the particular case that $i=0$, we have
\begin{align}\label{Eq:p-1-step-of-newman-i=0}
R^{p} \equiv R^{p-1} \cdot R &= \prod_{j=0}^{p-2}\left( \underbrace{(a^jba^{-j})}_{\beta_j}\underbrace{(a^{j+1}b^{-1}a^{-(j+1)}}_{\beta_{j+1}^{-1}})\underbrace{(a^p)^{-q}}_{\alpha^{-q}} \right) a^{p-1} \cdot bab^{-1} a^{-pq} \nonumber \\
&= \prod_{j=0}^{p-2}\left( \underbrace{(a^jba^{-j})}_{\beta_j}\underbrace{(a^{j+1}b^{-1}a^{-(j+1)}}_{\beta_{j+1}^{-1}})\underbrace{(a^p)^{-q}}_{\alpha^{-q}} \right) \underbrace{(a^{p-1}ba^{-(p-1)})}_{\beta_{p-1}} \underbrace{a^p}_{\alpha} \underbrace{(b^{-1})}_{\beta_0^{-1}} \underbrace{a^{-pq}}_{\alpha^{-q}},
\end{align}
where all equalities are in the free group on $A$. The right-hand side of \eqref{Eq:p-1-step-of-newman-i=0} is now a word over $\alpha$ and the $\beta_i$, so we find
\begin{equation}\label{Eq:newman-beta-defining-rewritten}
\tau(R^p) \equiv \prod_{j=0}^{p-2} \left( \beta_j \beta_{j+1}^{-1} \alpha^{-q} \right) \beta_{p-1} \alpha \beta_0^{-1} \alpha^{-q}.
\end{equation}
The reader may now readily verify, in exactly the same manner as above, that $\tau(a^i R^p a^{-i})$ is a cyclic conjugate of \eqref{Eq:newman-beta-defining-rewritten}; indeed, the conjugation will shift the indices on the $\beta_j$ to $\beta_{j+i}$, taken mod $p$, but the single $\alpha$ will still appear between the $\beta_{p-1}$ and the $\beta_0^{-1}$.

 In particular, $H$ is a one-relator group with the single defining relation $\tau(R^p)$. A cyclic conjugate of $\tau(R^p)$, putting it in a more readable form, is:
\[
w \equiv (\beta_0^{-1} \alpha^{-q} \beta_0)(\beta_1^{-1} \alpha^{-q} \beta_1) (\beta_2^{-1} \cdots )(\beta_{p-1}^{-1} \alpha^{-q} \beta_{p-1}) \alpha.
\]
Thus 
\begin{align*}
H &\cong \pres{Gp}{\alpha, \beta_0, \dots, \beta_{p-1}}{w = 1} \cong \pres{Gp}{\alpha, \beta_1, \dots, \beta_p}{\alpha = (\alpha^q)^{\beta_1} \cdots (\alpha^q)^{\beta_p}} \cong \NP(p,q)
\end{align*}
where, for the second isomorphism, we used the free group automorphism induced by $\alpha \mapsto -\alpha$, $\beta_i \mapsto \beta_i^{-1}$, shifting the indices of the $\beta_i$ up by one. Thus $H \cong \NP(p,q)$ is an index $p$ normal subgroup of $G$, as desired. 
\end{proof}

We now conclude that $\NP(p,q)$ has decidable Diophantine problem; indeed, if $p=1$ then the result is known \cite{Kharlampovich2020}). If $p>1$, then the result follows (e.g.) from Corollary~\ref{Cor:f.g.-ORwt-is-hyperbolic}. Alternatively, we can use the recently proved result that decidability of the Diophantine problem is inherited by finite index subgroups \cite{Levine2021}. Either way, we conclude:

\begin{corollary}\label{Cor:NP-have-dec-DP}
For every $p, q \geq 1$, the Diophantine problem is decidable in the group
\[
\NP(p,q) = \pres{Gp}{y, x_1, \dots, x_p}{y = (y^q)^{x_1} (y^q)^{x_2} \cdots (y^q)^{x_p}}.
\]
\end{corollary}

It would be interesting to know whether the subgroup and/or the submonoid membership problems in $\NP(p,q)$ are decidable.

{
\bibliography{BalancedBaumslag.bib} 
\bibliographystyle{plain}
}
 \end{document}